\documentclass[12pt]{amsart}

\usepackage{enumerate,amssymb,amsmath,mathrsfs}
\usepackage{amssymb,amsfonts,amsthm,amscd}
\usepackage{cite}
\usepackage[mathscr]{eucal}     % Use math script alphabet
\usepackage{graphicx}
\usepackage[arrow, matrix, curve]{xy}
\usepackage{color}
\usepackage[margin=4cm]{geometry}
\definecolor{gray}{gray}{.75}
\definecolor{gray2}{gray}{.50}

\usepackage{tikz}

\newcommand{\N}{\mathbb{N}}
\newcommand{\R}{\mathbb{R}}
\newcommand{\C}{\mathbb{C}}

\newcommand{\A}{\mathrm{\alpha}}
\newcommand{\w}{\mathrm{\omega}}

\newcommand{\e}{\epsilon}

\newcommand{\ff}{\mathrm{ff}}
\newcommand{\lf}{\mathrm{lf}}
\newcommand{\rf}{\mathrm{rf}}
\newcommand{\of}{\mathrm{of}}

\newcommand{\he}{\widehat{\eta}}

\newcommand{\phg}{\mathrm{phg}}
\newcommand{\Tr}{\mathrm{Tr}}

\newcommand{\wx}{\widetilde{x}}
\newcommand{\wy}{\widetilde{y}}
\newcommand{\wz}{\widetilde{z}}
\newcommand{\wt}{\widetilde{t}}

\newcommand{\ld}{\underline{\delta}}
\newcommand{\ud}{\overline{\delta}}

\newcommand{\del}{\partial}
\newcommand{\calU}{\mathcal U}
\newcommand{\calG}{\mathcal G}
\newcommand{\calH}{\mathcal H}
\newcommand{\calA}{\mathcal A}
\newcommand{\calE}{\mathcal E}
\newcommand{\calF}{\mathcal F}

\newcommand{\calC}{\mathcal C}
\newcommand{\calO}{\mathcal O}
\newcommand{\calV}{\mathcal V}

\newcommand{\calR}{\mathcal R}

\newcommand{\calD}{\mathcal D}
\newcommand{\RR}{\mathbb R}
\newcommand{\CC}{\mathbb C}
\newcommand{\NN}{\mathbb N}
\newcommand{\Ve}{\mathcal V_{\mathrm{e}}}
\newcommand{\Diffe}{\mathrm{Diff}_e}
\newcommand{\scrA}{\mathscr{A}}

\newcommand{\fdiag}{\mathrm{fdiag}}

\newcounter{mnotecount}[section]

\newcommand{\dbar}{{\mathchar'26\mkern-11mu\mathrm{d}}}

\newtheorem{thm}{Theorem}[section]
\newtheorem{cor}[thm]{Corollary}
\newtheorem{remark}[thm]{Remark}
\newtheorem{prop}[thm]{Proposition}
\newtheorem{lemma}[thm]{Lemma}
\newtheorem{defn}[thm]{Definition}

\newtheorem{assumption}[thm]{Assumption}

\numberwithin{equation}{section}

\begin{document}

\title[Elliptic theory of differential edge operators, II]
{Elliptic theory of differential edge operators, II: boundary value problems}

\author{Rafe Mazzeo}
\address{Stanford University \\ Department of Mathematics \\
450 Serra Mall \\ Stanford, CA 94305-2125\\ USA}
\email{mazzeo@math.stanford.de}

\author{Boris Vertman}
\address{Bonn University \\ Department of Mathematics \\
Endenicher Allee 60 \\ Bonn, 53115\\ Germany}
\email{vertman@math.uni-bonn.de}

\thanks{The first author was supported in part by the NSF grants DMS-0805529 and 1105050}
\thanks{The second author was supported by the Hausdorff Research Institute and DFG}

\date{\today}

\begin{abstract}
This is a continuation of the first author's development \cite{M} of the theory of elliptic differential operators
with edge degeneracies. That first paper treated basic mapping theory, focusing on semi-Fredholm properties
on weighted Sobolev and H\"older spaces and regularity in the form of asymptotic expansions of solutions.
The present paper builds on this through the formulation of boundary conditions and the construction of 
parametrices for the associated boundary problems. As in \cite{M}, the emphasis is on the geometric 
microlocal structure of the Schwartz kernels of parametrices and generalized inverses. 
\end{abstract}

\maketitle
%\tableofcontents
\pagestyle{myheadings}
\markboth{\textsc{Elliptic theory of differential edge operators, II}}{\textsc{Rafe Mazzeo and Boris Vertman}} 

%%%%%%%%%%%%%%%%%%%%
\section{Introduction}
%%%%%%%%%%%%%%%%%%%%
Degenerate elliptic operators on manifolds with boundary or corners arise naturally in many different problems in
partial differential equations, geometric analysis, mathematical physics and elsewhere. Over the past 
several decades, many types of such equations have been studied, often by ad hoc methods but sometimes 
through the development of a more systematic theory to handle various classes of operators. One particularly
fruitful direction concerns the elliptic operators associated to (complete or incomplete) iterated edge metrics 
on smoothly stratified spaces. The simplest examples of such operators include the Laplace
operators for spaces with isolated conic singularities or with asymptotically cylindrical ends. Other 
important special cases include nondegenerate elliptic operators on manifolds with boundaries
or the Laplacians on asymptotically hyperbolic (conformally compact) manifolds. The class of elliptic 
operators on spaces with simple edge singularities includes both of these sets of examples. A final
example is the Laplacian (or any other elliptic operator) on a smooth manifold written in Fermi coordinates around 
a smooth embedded submanifold. 

To be more specific, let $M$ be a compact manifold with boundary, and suppose that $\del M$ is the total
space of a fibration with base $B$ and fibre $F$. Choose coordinates $(x,y,z)$ near the boundary so
that $x = 0$ defines $\del M$, $y$ is a set of coordinates on $B$ lifted to $\del M$ and then extended
into $M$ and $z$ are independent functions which restrict to a coordinate system on each fibre $F_y$. 
A differential operator of order $L$ is called an edge operator of order $m$ if it has the form
\[
L = \sum_{j + |\alpha| + |\beta| \leq m} a_{j \alpha \beta}(x,y,z) (x\del_x)^j (x\del_y)^\alpha \del_z^\beta.
\]
We assume that the coefficients $a_{j \alpha \beta}$ are all $\calC^\infty$ on the closed manifold $M$; these
can either be scalar or (if $L$ acts between sections of vector bundles) matrix-valued.  We say that $L$
is edge elliptic if it is an elliptic combination of the constituent vector fields $x\del_x$, $x\del_{y_i}$
and $\del_{z_j}$; an invariant definition is indicated in \S 2.  The examples above 
all fall into this class, or else are of the form $x^{-m}L$ where $L$ has this form; operators of this latter
sort are called `incomplete' edge operators since they include Laplacians of metrics with incomplete
edge singularities. 

The present paper is the continuation of a now rather old paper by the first author \cite{M} which develops a
framework for the analysis of elliptic differential edge operators based on the methods of geometric 
microlocal analysis.  That paper establishes many fundamental results concerning the mapping properties of
these operators and the regularity properties of solutions, and those results have had very many applications,
both in analysis and geometry, in the intervening years.  We review this theory in \S 2. The mapping properties 
considered there were for an elliptic edge operator acting between weighted Sobolev and H\"older spaces. This 
left open, however, any development of a more general theory of ``elliptic edge boundary value problems''. 
The present paper finally addresses this aspect of elliptic edge theory.

An important generalization involves the study of elliptic operators with iterated edge singularities; 
examples include Laplacians on $\calC^\infty$ polyhedra or conifolds, see \cite{MM}, as well as on general 
smoothly stratified spaces \cite{ALMP}, \cite{ALMP2}.   As of yet, there is no complete elliptic theory 
in this general setting, although many special cases and specific results have been obtained by a variety of authors.
Certainly the closest to what we do here is the work of Gil, Krainer and Mendoza \cite{GKM} and the recent and ongoing 
work of Krainer-Mendoza \cite{KM}, \cite{KM2}. In various respects, these last papers go much farther than what 
we do here. We mention also the theory developed by Schulze, 
\cite{S1}, \cite{S2}, \cite{S3}. The emphasis in those papers and monographs is the development of a 
hierarchy of algebras of pseudodifferential operators, structured as in \cite{BdM}, with emphasis on the 
operator-symbol quantization associated to spaces with both simple and iterated edge singularities. Other notable 
contributions include the work of Maz'ya and his collaborators, see \cite{KMR}, as well as Nistor \cite{Ni}, 
Ammann-Lauter-Nistor \cite{ALN} and Br\"uning-Seeley \cite{BS}.   As noted above, Krainer and Mendoza \cite{KM}, \cite{KM2} 
also treat edge boundary problems, and in fact do so for more general operators with variable indicial roots, but their methods 
are somewhat different from the ones here. These last authors have very generously shared some important ideas 
from their work before the appearance of \cite{KM}, described here in \S 3, which form a necessary part of our analysis. 

There are many reasons for developing a theory of more general boundary conditions for elliptic edge operators, in 
particular from the point of view developed here. Perhaps most significant is the importance of mixed or global 
boundary conditions, either of local (Robin) or Atiyah-Patodi-Singer type, in the study of index theory for generalized 
Dirac operators on spaces with simple edge singularities, all of which appear in many natural problems. Similarly, the 
study of the eta invariant and analytic torsion for Dirac-type operators on spaces with various boundary conditions 
on spaces with isolated conic singularities has proved to be quite interesting. All of these directions fall within 
the scope of one or more of the other approaches cited above.  

The geometric microlocal methods used here have a distinct advantage over other (e.g.\ more directly Fourier analytic) 
approaches: our primary focus is on Schwartz kernels rather than abstract mapping properties or methods too
closely tied to more standard pseudodifferential theory, and because of this it is equally easy to obtain results 
adapted to any standard types of function spaces that one might wish to use, e.g.\ weighted H\"older or $L^p$ spaces. 
This transition between mapping properties on different types of spaces seems more difficult using those other 
approaches, although having such properties available is quite important when studying nonlinear geometric 
problems on spaces with edge singularities, cf.\ \cite{JMR} for a recent example. 

One limitation of the current development, however, is that we do not treat the delicate regularity issues associated 
with the possibility of smoothly varying indicial roots.  As in \cite{M}, we make a standing assumption that all 
operators considered here have constant indicial roots, at least in the critical weight-range $(\ld, \ud)$. 
We refer to \S 2 for a description of all of this.

Because their precise description requires a number of preliminary definitions, we defer to \S 3 a careful description 
of our main results; however, we now state them briefly and somewhat
informally.  The starting point is the basic statement that if $L$ is an elliptic edge operator, as above, then under appropriate
hypotheses on the indicial roots and assuming the unique continuation property for the reduced Bessel operator
$B(L)$, see \S 2 for these, the mapping
\[
L: x^\delta H^m_e(M) \longrightarrow x^\delta L^2(M)
\]
is essentially surjective, i.e.\ has closed range with finite dimensional cokernel, provided 
$\delta \leq \ld$ and $\delta$ is nonindicial, and is essentially injective, i.e.\ has
closed range with finite dimensional nullspace, if $\delta \geq \ud$. 
Suppose that $u \in x^{\ld} H^m_e$ and $Lu = f \in x^{\ud}$. Then it is proved in \cite{M} that
\[
u \sim \sum_{j=1}^N \sum_{\ell, p \in \mathbb N_0} u_{j p \ell}(y,z) x^{\gamma_j + \ell} (\log x)^p + \tilde{u}, 
\]
where the sum is over all indicial roots (see \S 2) of $L$ and indices $p, \ell$ such that $\gamma_j + \ell \in (\ld, \ud)$
and $p$ is no greater than some integer $N _{j,\ell}$, and where $\tilde{u} \in x^{\ud}H^m_e$.  The 
subcollection of leading coefficients $\{u_{j, N_{j 0}, 0}\}$ is called the Cauchy data of $u$ and denoted 
$\calC(u)$.  A boundary condition for this edge problem consists of a finite collection $Q = \{Q_{kj}\}$ 
of pseudodifferential operators acting on these leading coefficients. (Since the precise formulation is somewhat 
intricate, we defer this for now.) We then study the mappings:
\[
\begin{aligned}
& L u = f \in x^{\ld}L^2  \\
& Q( \calC(u)) = \phi.
\end{aligned}
\]
The main results here give conditions for when this mapping is Fredholm or semi-Fredholm acting on appropriate
weighted Sobolev spaces.  We follow the methods due originally to Calderon, described particularly well in 
the monograph of Chazarain and Piriou \cite{CP}, and later extended significantly by Boutet de Monvel and others; however, we
do not define the full Boutet de Monvel calculus in this edge setting. We also give the precise structure of the 
Schwartz kernel of the generalized inverse of this mapping, and consequently can study this problem on other 
function spaces. We do not treat any application of these results in this paper, but must rely on the reader's 
knowledge of the centrality of elliptic boundary problems in the standard setting, and on his or her faith 
that this extension of that theory will also have broad applicability.

\bigskip

{\bf Acknowledgements.} The first author's understanding of elliptic boundary problems reflects his long and fruitful
interactions with Richard Melrose, and also Pierre Albin, Charlie Epstein, Gerd Grubb, Thomas Krainer, 
Gerardo Mendoza and Andras Vasy. The second author would like to express his gratitude to Bert-Wolfgang Schulze for 
encouragement and many useful discussions concerning his alternate formulation of edge calculus. He also 
wishes to acknowledge many useful discussions with Andras Vasy. The second author also thanks the 
Department of Mathematics at Stanford University for its hospitality during a major part of the research and 
writing which led to this paper. Both authors offer special thanks to Thomas Krainer and Gerardo Mendoza for many
useful discussions on this subject, and in particular for explaining their theory of trace bundles to us before the appearance of \cite{KM}. 
We hope the reader will regard their work, as we do, as a good counterpart to ours, with somewhat different aims.

%\tableofcontents

%%%%%%%%%%%%%%%%%%%%
\section{A review of the edge calculus} \label{fundamentals}
%%%%%%%%%%%%%%%%%%%%
We begin by recalling in more detail the geometric and analytic framework necessary to discuss the theory of 
differential and pseudodifferential edge operators, and then review the main theorems from \cite{M} concerning
the semi-Fredholm theory and asymptotics of solutions.  This section is meant as a brief review, and is not meant
to be self-contained. We refer the reader to \cite{M} for elaboration and proofs of all the definitions and facts 
presented here. 

\medskip

\noindent{\bf Edge structures}
As in the introduction, let $M$ be compact manifold with boundary, and suppose that $\del M$ is the total space of a fibration 
$\phi:\partial M \to B$ with fibre $F$.  We set $b = \dim B$ and $f = \dim F$. 

The fundamental object in this theory is the space $\Ve$ of all smooth vector fields on $M$ which are unconstrained in 
the interior and which are tangent to the fibres of  $\phi$ at $\del M$; clearly $\Ve$ is
closed under Lie bracket.  We shall routinely use local coordinate systems near the boundary of the following form: 
$x$ is a defining function for the boundary (i.e.\ $\del M = \{x = 0\}$), $y_1, \ldots, y_b$ is a set of local coordinates
on $B$ lifted to $\del M$ and then extended into $M$, and $z_1, \ldots, z_f$ is a set of independent functions
which restricts to a coordinate system on each fibre $F_y$. In terms of these, 
\begin{equation}
\Ve = \mathrm{Span}_{\calC^\infty} \, \{ x\del_x, x\del_{y_1}, \ldots, x \del_{y_b}, \del_{z_1}, \ldots, \del_{z_f} \}.
\label{v-e}
\end{equation}
In other words, any $V \in \Ve$ can be expressed locally as 
\[
V = a x\del_x + \sum b_i x \del_{y_i} + \sum c_j \del_{z_j}, \quad \mbox{where}\ a, b_i, c_j \in \calC^\infty(\overline{M}). 
\]

Any differential operator can be expressed locally as the sum of products of vector fields, and so we can define
interesting subclasses of operators by restricting the vector fields allowed in these decompositions. In particular,
define $\Diffe^*(M)$ to consist of all differential operators which are locally finite sums of products of
elements in $\Ve$. With the subscript corresponding to the usual order filtration, we have, in local coordinates,
\begin{equation}
\Diffe^m(M) \ni L = \sum_{j+|\alpha|+|\beta| \leq m}  a_{j \alpha \beta}(x,y,z) (x\del_x)^j (x\del_y)^\alpha \del_z^\beta,
\label{diffem}
\end{equation}
with all $a_{j \alpha \beta} \in \calC^\infty$. 
Here and later we use standard multi-index notation to describe (differential) monomials. If $L$ acts between sections
of two bundles $E$ and $F$, then taking local trivializations of these bundles, the coefficients here are matrix-valued. 

There is a natural edge tangent bundle ${}^eTM$ defined by the property that $\Ve$ coincides with its {\it full} space 
of $\calC^\infty$ sections; its dual is the edge cotangent bundle ${}^eT^*M$, which has a local $\calC^\infty$ basis of 
sections consisting of the $1$-forms
\[
\frac{dx}{x}, \frac{dy_1}{x}, \ldots, \frac{dy_b}{x}, dz_1, \ldots, dz_f.
\]

Any $L \in \Diffe^m(M)$ has symbol
\[
{}^e\sigma_m(L)(x,y,z, \xi, \eta, \zeta) = \sum_{j+|\alpha|+|\beta| = m} a_{j \alpha \beta} (x,y,z) \xi^j \eta^\eta \zeta^\beta,
\]
which is well-defined as a smooth function on ${}^eT^*M$ which is a homogenous polynomial of degree $m$ on
each fibre.  If $L$ acts between sections of two vector bundles $E$ and $F$, then ${}^e\sigma_m(L)$ takes values
in $\mathrm{End}(\pi^*E, \pi^* F)$, where $\pi: {}^e T^*M \to M$.   The operator $L$ is said to be elliptic (in the
edge sense) if this symbol is invertible when $(\xi,\eta,\zeta) \neq (0,0,0)$. 

A (complete) edge metric is a smooth positive definite section of $\mathrm{Sym}^2 ({}^e T^*M)$. It is not hard to
check that if $g$ is any metric of this type, then its scalar Laplacian, Hodge Laplacian, and all other natural
elliptic geometric operators (e.g. the rough Laplacian, the Lichnerowicz Laplacian, twisted Dirac operators, etc.)
are all elliptic edge operators (N.B.; some of these operators are of this type only if expressed in terms of an 
appropriate basis of sections of the bundles on which they act). 
Similarly, an incomplete edge metric $g$ is one of the form $x^2 \tilde{g}$, where
$\tilde{g}$ is a complete edge metric.  Its Laplacian is of the form $x^{-2}L$ where $L \in \Diffe^2(M)$, and there
are analogous assertions for the other elliptic operators mentioned above.  In practice one often restricts
to a smaller class of metrics (for example, requiring that $g$ does not contain the term $x^{-1}dxdy$, though
even even more rigid hypotheses arise naturally), see \cite{MV} for more on this.

\medskip
\noindent{\bf Model operators}
Let $L$ be an elliptic edge operator of order $m$, expressed as in \eqref{diffem}. The analysis of the mapping 
properties of $L$ relies on a variety of associated model operators. 

First, the principal edge symbol ${}^e\sigma_m(L)$ is a purely algebraic model for $L$ at any point; the microlocal
inversion of $L$, uniformly up to the boundary, relies on the invertibility of this object. 

Next, associated to every point $y_0 \in B$ (taken as the origin in the $y$ coordinate system) is the {\it normal operator} 
\begin{equation}
N_{y_0}(L) = \sum_{j+|\alpha| + |\beta| \leq m} a_{j \alpha \beta}(0, 0, z) (s\del_s)^j (s\del_u)^\alpha \del_z^\beta;
\label{normalop}
\end{equation}
this acts on functions on the model space $\RR^+_s \times \RR^b_u \times F_{y_0}$, where $s$ and $u$ are global
{\it linear} variables on a half-space which can be regarded as being the part of the tangent bundle which is the
inward normal to $F_{y_0}$. This model space is naturally identified with the tangent cone with respect
to the family of dilations $(x,y,z) \mapsto (\lambda x, \lambda y, z)$ as $\lambda \nearrow \infty$.

Another operator which models the behaviour of $L$ near $F_{y_0}$ is the Bessel operator 
\begin{equation}
B_{y_0, \hat{\eta}}(L) = \sum_{j + |\alpha| + |\beta| \leq m}  a_{j \alpha \beta} (0,0,z)  (t\del_t)^j (-i t \hat{\eta})^\alpha \del_z^\beta.
\label{Bessel}
\end{equation}
Here $\hat{\eta} = \eta/|\eta|$, where $\eta \in T_{y_0}^*B$ (i.e.\ $\hat{\eta}$ lies in the spherical conormal bundle
$S_{y_0}^*B$).  This is obtained from $N(L)$ by first passing to the Fourier transform in $u$ (which transforms $s \del_u$ 
to $-i s \eta$) and then rescaling by setting $t = s |\eta|$. These operations are reversible so the family $B_{y_0, \hat{\eta}}(L)$
is completely equivalent to $N_{y_0}(L)$ even though it appears to be simpler.  
Note that the structure of $B(L)$ as $t \to \infty$ captures the behaviour as $|\eta| \to \infty$, 
and hence corresponds to local behaviour for $N(L)$. 

Finally, the indicial operator, which is an elliptic $b$-operator in the sense of \cite{Mel-APS}, 
\cite{M}, on $\RR^+ \times F$, is defined by
\begin{equation}
I_{y_0}(L) = \sum_{j+|\beta| \leq m} a_{j 0 \beta} (0, 0, z) (t\del_t)^j \del_z^\beta.
\label{indicial}
\end{equation}
This is obtained from the Bessel-normal operator by dropping the terms which are lower order in the $b$-theory.
Thus $B_{y_0,\hat\eta}(L) = I_{y_0}(L) + E$, where $E$ is truly lower order on any finite interval $0 < t < t_0$. However, 
as remarked above, the large $t$ behaviour of $B_{y_0,\hat\eta}(L)$ contains important information missing in the indicial operator. 

\medskip

\noindent{\bf Indicial roots and the trace bundle} The indicial operator can be conjugated, via the Mellin transform, 
to the {\it operator pencil},
\begin{equation}
I_{y_0}(L)(\zeta) := \sum_{j+|\beta| \leq m} a_{j 0 \beta} (0, 0, z)(-i \zeta)^j \del_z^\beta, 
\label{indfamily}
\end{equation}
which depends smoothly on $y_0 \in B$; this is often called the indicial family of $L$.  (An operator pencil, an important 
generalization of a resolvent family,  is simply a polynomial family with operator coefficients.)   Because the 
coefficient of $\zeta^j$ has order $m-j$, and the coefficient of $\zeta^0$ is an elliptic operator on $F_{y_0}$ 
of order $m$. Hence the analytic Fredholm theorem may be applied, and this implies that this family is either never 
invertible, for any $\zeta 
\in \CC$, or else that this inverse is meromorphic, and the Laurent coefficients at each pole are operators of finite 
rank.  It is certainly necessary to assume that $I_{y_0}(L)(0) = \sum_{|\beta| = m} a_{0 0 \beta}(z) \del_z^\beta$ has 
index zero, otherwise we are necessarily in the first, nowhere invertible, case.  A standard condition to ensure 
that the inverse exists at one point, and hence away from a discrete set, is that the resolvent of the (ordinary) symbol, 
$(\sigma_m( I_{y_0}(L)(0)) - \lambda)^{-1}$, satisfy standard elliptic symbol-with-parameter estimates 
in some open conic sector.  In any case, we shall assume that we are in some setting which allows
us to conclude that the indicial family has meromorphic inverse.   
\begin{defn}
The boundary spectrum of $L$, $\mbox{Spec}_b(L)$, is the set of locations of the poles of the meromorphic family
$I_{y_0}(L)^{-1}$ (this is also called the spectrum of the operator pencil); elements of $\mbox{Spec}_b(L)$ are called 
the indicial roots of $L \ \textup{(at $y_0$)}$. 
\end{defn}
We have tacitly suppressed the fact that these indicial roots may vary with $y_0$.  The analysis of edge operators 
with indicial roots depending nontrivially on $y_0$ is an interesting and difficult topic, and is discussed
in detail in the forthcoming work of Krainer and Mendoza. However, partly because this behaviour often does not
occur for the natural examples of edge operators, we choose to make the basic 
\begin{assumption}[Constancy of indicial roots]
\label{discrete}
The spectrum of the indicial family is discrete and the location of the poles of $I_{y_0}(L)(\zeta)^{-1}$ does 
not depend on $y_0 \in B$. 
\label{assump-const-ind-rts}
\end{assumption}

It is necessary to make one further hypothesis:
\begin{assumption}
For each $(y_0, \hat\eta)$, the Bessel operator $B_{y_0,\hat\eta}(L)$ is injective on
$t^{\delta}L^2(dt dz)$ for $\delta \gg 0$, and is surjective when $\delta \ll 0$. 
\label{assump-uniq-cont}
\end{assumption}
This holds in many interesting situations, see \cite{Ma-UC}. 

We henceforth always work with operators satisfying both of these assumptions. In the following, we 
fix two nonindicial values $\ld$ and $\ud$ such that $B(L)$ is injective on $t^{\ud}L^2$ and surjective
on $t^{\ud} L^2$.  We then write $\mathfrak{S}(L)$ for the set of all indicial roots in this critical strip,
omitting their multiplicity: 
\[
\mathfrak{S}(L):=\{\zeta_j \in \C: \ld -1/2 < \Im\zeta_0 < \ud - 1/2\ \mbox{and}\ \exists \, p \in \N\ 
\mbox{s.t.}\ (\zeta_j,p) \in \textup{Spec}_b(L)\}.
\]
The shift by $1/2$ appears because we are using the measure $dt dz$.

\begin{prop}
Let $\w \in t^{\ld} L^2( dt dz)$ and $I_{y_0}(L) \w = f \in t^{\ud}L^2(dt dz)$. Then
\[
\w =  \sum_{\zeta_j \in \mathfrak{S}(L)} \, \sum_{p = 0}^{p_j}  \w_{j, p}(z) t^{-i\zeta + \ell} (\log t)^p + \widetilde{\w}, 
\quad  \w_{j,p} \in \calC^\infty(F),\ \widetilde{\w} \in t^{\ud}L^2(dt dz). 
\]
Here $p_j + 1$ is equal to the order of the pole of $I_{y_0}(L)^{-1}$ at $\zeta_j$. 
\label{explsoln}
\end{prop}

Dropping the subscript $y_0$, we pass to the Mellin transform of this equation (see \cite{M}),
which is $I(L)(\zeta) \w_M(\zeta, z) = f_M(\zeta, z)$. The Mellin transforms $\w_M(\zeta, z)$ and $f_M(\zeta,z)$ are 
holomorphic in $\Im \zeta < \ld - 1/2$ and $\Im \zeta < \ud - 1/2$, respectively, both with values in $L^2(F; dz)$.
Thus $\w_M = - I(L)(\zeta)^{-1} f_M$ extends meromorphically to this larger half-plane.  Taking the inverse Mellin
transform by integrating along the line $\Im \zeta = \delta < \ld - 1/2$ and then shifting the contour across the
poles in this horizontal strip produces the expansion. 

Using this result, we can define the trace of $\w$ to consist of the set of functions $\w_{j,p}$ over all $\zeta_j
\in \mathfrak{S}(L)$ and $p \leq p_j$. There is a subtlety here in that although we require that each $\zeta_j$
is independent of $y_0$, the same may not be true of the $p_j$, so we need to explain carefully the sense in which
this expansion depends smoothly on $y_0$.  

To describe this we need make a small detour. First note that the algebraic multiplicity of each pole is well-defined. 
As in the special case of the resolvent family of a non self-adjoint operator, this algebraic multiplicity is a 
positive integer which measures the dimension of the space of generalized eigenvectors associated to that 
indicial root.  Since there are several different-looking (but equivalent) definitions of this quantity,
we provide a slightly longer description than strictly necessary, for the reader's convenience. For simplicity of notation, 
omit the dependence on $y_0$ for the moment, and suppose that $\zeta_0$ is the indicial root in question.  
The geometric eigenspace of $I(L)(\zeta_0)$ is the subspace of $\calC^\infty(F)$ consisting of all $\phi_0$ such that 
$I(L)(\zeta_0)\phi_0 = 0$, and its dimension is called the geometric multiplicity of the indicial root.  
Suppose now that there exist additional functions $\phi_j \in \calC^\infty(F)$, $j = 1, \ldots, k-1$ such that 
\begin{equation}
\sum_{j=0}^{\ell} \frac{1}{j!} (\del_\zeta^j I(L))(\zeta_0) \phi_j  = 0, \ \ \ell = 1, \ldots, k-1;
\label{jordanchain}
\end{equation}
this sequence of equations is equivalent to the single condition
\[
I(L)(\zeta)( \phi_0 + (\zeta-\zeta_0) \phi_1 + \ldots (\zeta-\zeta_0)^{k-1} \phi_{k-1}) = \calO(|\zeta-\zeta_0|^k).
\]
The ordered $k$-tuple $(\phi_0, \ldots, \phi_{k-1})$ is called a generalized eigenvector, and the maximal length of all such 
chains beginning with $\phi_0$ is called the multiplicity of $\phi_0$ and denoted $m(\phi_0)$.  In other words, $m(\phi_0)$ 
measures the order to which $\phi_0$ can be extended as a formal series solution of $I(L)(\zeta) \sum 
(\zeta-\zeta_0)^j \phi_j \equiv 0$. We refer to $\phi_0(\zeta)= \sum (\zeta-\zeta_0)^j \phi_j$ as
the root function associated to the eigenvector $\phi_0$. Following \cite[\S 1.1]{KMR}, choose a basis 
$\{\phi_{0,1}, \ldots, \phi_{0,N}\}$ for the geometric eigenspace of $I(L)(\zeta_0)$ so that $m(\phi_{0,1}) \leq 
m(\phi_{0,2})\leq \cdots \leq m(\phi_{0,N})$, and then define the algebraic multiplicity of the pole to be the number 
\[
m(\zeta_0) = \sum_{j=1}^N m(\phi_{0,j}). 
\]

There is an alternate description, following \cite[Ch. XI]{GGK}, which gives a slightly different intuition for this number.
The first step is to write the holomorphic family $I(L)(\zeta)$ as the product $E(\zeta)D(\zeta) F(\zeta)$, locally 
near $\zeta_0$. Here $E(\zeta)$ and $F(\zeta)$ are holomorphic and invertible for $|\zeta-\zeta_0| < \epsilon$ and 
\[
D(\zeta) = P_0 + (\zeta- \zeta_0)^{\kappa_1} P_1 + \ldots + (\zeta- \zeta_0)^{\kappa_r} P_r,
\]
where the $P_j$ are mutually disjoint projectors (i.e.\ $P_i P_j =0$ if $i \neq j$), $P_0$ has infinite rank, 
$\mbox{rank} \, P_j = 1$, $j > 0$, and $P_0 + \ldots + P_r = \mbox{Id}$. Clearly $\kappa_r$ is the order of 
the pole of $I(L)(\zeta)^{-1}$ at $\zeta_0$, and a straightforward calculation shows that the algebraic multiplicity 
$m(\zeta_0)$ is equal to $\kappa_1 + \ldots + \kappa_r$.  

In any case, by an extension of the theorem of Keldy\u{s} \cite{Kel1},\cite{Kel2}, see Gohberg-Sigal \cite{GohSig} and 
Menniken-M\"oller \cite{MM}, the generalized eigenvectors characterize the singular part of the Laurent expansion
of $I_{y_0}(L)(\zeta)^{-1}$ at $\zeta_0$, as follows. These sources prove that there exists a set of polynomials in $\zeta$, 
$(\psi_1(\zeta), ..,  \psi_N(\zeta))$, taking values in $\calD'(F)$ (distributions on $F$), and an operator-valued family $H(\zeta)$ 
which  is holomorphic near $\zeta_0$, such that 
\begin{align}\label{representation}
I_{y_0}(L)(\zeta)^{-1} = \sum_{j=1}^N (\zeta - \zeta_0)^{-k_j} \phi_{0,j}(\zeta) \otimes \psi_j(\zeta) + H(\zeta).
\end{align}
Here, for each $j$, $\phi_{0,j}(\zeta)$ is the root function corresponding to the element $\phi_{0,j}$ in the geometric
eigenspace and $k_j = m(\phi_{0,j})$. This implies that for any holomorphic function $u(\zeta)$ taking 
values in $\calC^\infty(F)$, each singular Laurent coefficient of $I_{y_0}(L) (\zeta)^{-1} u(\zeta)$ at $\zeta_0$ 
is a linear combinations of the coefficients $\phi_{\ell, j}$ of the root functions $\phi_{0,j}(\zeta) = \sum \phi_{\ell,j} (\zeta -
\zeta_0)^\ell$.  

Going back to Proposition~\ref{explsoln}, and using the notation there, it is clear that each of the singular Laurent 
coefficients of $\w_M$ at any pole $\zeta_0$ in the horizontal strip is linear combination of coefficients of the root 
functions for $I(L)(\zeta_0)$, hence the coefficients of the terms $t^{-i\zeta_0} (\log t)^p$ in the partial expansion for $\w$
are constituted by these same functions.

We now finally come to the issue of smooth dependence on $y_0$. The key fact is that the algebraic multiplicity $m(\zeta_0)$ 
of each pole is invariant under small perturbations, and hence is locally independent of $y_0$. This follows from an 
operator-valued version of Rouch\'e's theorem; we refer to \cite[\S 1.1.2]{KMR} for a more careful description, 
and to \cite[Theorem 9.2]{GGK} for a proof.  

In fact, slightly more is true: the direct sum of the coefficients of the root functions $\{\phi_{0,j}(\zeta), j \leq N\}$
form a vector space $\calE_{ y_0}(\zeta_0)$ of dimension $m(\zeta_0)$, and as $y_0$ varies, these vector spaces fit 
together as a smooth vector bundle $\calE(\zeta_0) = \calE(L; \zeta_0)$ over each connected component of $B$.  

To define this bundle, write $\mathscr{M}(\zeta_0)$ and $\mathscr{H}(\zeta_0)$ for the spaces of germs of meromorphic 
and holomorphic functions, respectively, at $\zeta_0$.  Following the definitions above, if $u(\zeta)$ is a holomorphic 
$\calC^\infty(F)$-valued function defined near $\zeta_0$, then the Laurent coefficients of $I_{y_0}(L)(\zeta)^{-1} u(\zeta)$ 
at $\zeta_0$ lie in $\calE_{y_0}(\zeta_0)$. We take this as our primary definition and hence let 
\[
\calE_{y_0}(\zeta_0) = \{ [u] \in \mathscr{M}(\zeta_0) / \mathscr{H}(\zeta_0) : [I_{y_0}(L)(\zeta)^{-1} u]=0\};
\]
equivalently, $\calE_{ y_0}(\zeta_0)$ is identified with the kernel of $I_{y_0}(L)$ on the space of all finite combinations 
$\sum_q  a_{j, q}(z) t^{-i\zeta_0 }(\log t)^q$ with $a_{j,  q} \in \calC^\infty(F)$. 
\begin{prop}[Krainer and Mendoza \cite{KM}]
\[
\calE(L;\zeta_0) :=\coprod\limits_{y_0 \in B} \calE_{y_0}(L;\zeta_0) \overset{\pi}\longrightarrow B
\]
is a smooth vector bundle of rank $m(\zeta_0)$. 
\label{trace-bundle}
\end{prop}
We sketch some elements of the proof (recalling however that those authors work in the more general setting where the 
$\zeta_j$ may also vary with $y$). For each $y_0 \in B$, and indicial root $\zeta_0 \in \mathfrak{S}(L)$, Krainer and Mendoza 
construct an independent set of smooth functions $\{\phi_{y_0, j}\}_{j=1}^{m_0}$, $m_0=m(\zeta_0)$, in a neighbourhood $\calU$ of 
$y_0$, which form a basis of $\calE_{y}(\zeta_0)$ for each $y\in \calU$. They then show that if $\phi(t,y,z) \in t^{\ld}L^2$ 
depends smoothly on $y \in \calU \subset B$ and $z \in F$, and $I_y(L) \phi \equiv 0$, then there exist smooth functions 
$f_j:\calU\to \CC$, $j=1,.., m_0$ such that 
\[
\phi(t,y,z) = \sum_{j=1}^{m_0} f_j(y) \phi_{y,j}(t,z).
\]
It follows from this that $\calE(L; \zeta_0)$ is a smooth vector bundle over $B$.  A nonobvious 
consequence is that the ranges of the various singular Laurent coefficients of $I_{y}(L)(\zeta)^{-1}$ remain 
independent of one another as $y$ varies. 

To conclude, let us remark that the full strength of Assumption~\ref{assump-const-ind-rts} is not needed: it is 
only necessary that every indicial root with real part in the critical interval $[\ld, \ud]$ for any $y_0$ is independent 
of $y_0$, so in particular, there are no indicial roots with imaginary parts crossing the levels $\ud$, $\ld$.  We shall 
phrase most results as if all indicial roots are constant, but remark at various points how 
results change in this slightly more general setting. 
  
\medskip

\noindent{\bf Mapping properties} 
Each of the model operators described above plays an important role in determining the refined mapping properties of $L$. 
The basic result, stated more carefully below, is that if both ${}^e\sigma_m(L)$ and $N(L)$ are invertible (as a bundle map and
as an operator between weighted Sobolev spaces, respectively); we encompass this pair of properties by saying 
that $L$ is {\it fully elliptic} -- then $L$ itself is Fredholm between the analogous weighted Sobolev spaces. For this reason, 
the pair $({}^e \sigma_m(L), N(L))$ should be regarded as the full symbol of $L$.  This is the simplest nontrivial case of a symbol 
hierarchy for iterated edge structures (as in Schulze's work). 

We shall let $L$ act on weighted Sobolev and H\"older spaces.  Fix a reference measure $dV = dx dy dz$ (more precisely,
$dV$ is a smooth, strictly positive multiple of Lebesgue measure). For any $k \in \mathbb N_0$, define
\[
H^k_e(M) = \{u:  V_1 \ldots V_\ell u \in L^2(dV)\ \forall\, V_j \in \Ve\ \mbox{and}\ \ell \leq k\}.
\]
Using interpolation and duality (or using edge pseudodifferential operators) one also defines $H^s_e(M)$ for any $s \in \RR$. 
We also define their weighted versions
\[
x^\delta H^s_e(M) = \{ u = x^\delta v:  v \in H^s_e(M)\}. 
\]
Note that these are the Sobolev spaces associated to any complete edge metric $g$ (though the measure $dV$ is equal
to $x^{b+1}$ times the Riemannian density for such a metric).  

Similarly, we define the H\"older seminorm
\[
[ u ]_{e; 0,\alpha} = \sup_{ (x,y,z) \neq (x',y',z')} \frac{ |u(x,y,z) - u(x',y',z')|(x+x')^\alpha}{ |x-x'|^\alpha + |y-y'|^\alpha + 
(x+x')^\alpha |z-z'|^\alpha}.
\]
This is simply the standard H\"older seminorm associated to the Riemannian distance associated to the complete metric $g$.
The edge H\"older space $\Lambda^{0,\alpha}_e(M)$ consists of functions $u$ such that $\sup |u| + [ u ] _{e, 0,\alpha} < \infty$.
We also define the weighted edge H\"older spaces
\[
x^\delta \Lambda^{k,\alpha}_e(M) = \{ u = x^\delta v:  V_1 \ldots V_\ell v \in \Lambda^{0,\alpha}_e(M)\ \ell \leq k\ \mbox{and}\ 
V_j \in \Ve\}.
\]

It is clear from the definitions that if $L \in \Diffe^m(M)$, then 
\begin{eqnarray}
& L: x^\delta H^s_e(M) &\longrightarrow x^\delta H^{s-m}(M)  \label{Lsob} \\
& L: x^\delta \Lambda^{k+m,\alpha}_e(M) &\longrightarrow x^\delta \Lambda^k_e(M) \label{Lhold}
\end{eqnarray}
are bounded mappings for every $\delta, s \in \RR$ and $k \in \mathbb N_0$. This is \cite[Cor. 3.23]{M}. 

The basic and most important mapping property for elliptic edge operators is the following. 
\begin{prop}(\cite[Thm. 6.1]{M}) Suppose that $L \in \Diffe^m(M)$ is elliptic satisfying the
Assumption~\ref{assump-const-ind-rts}, and that $\delta \notin \mbox{Spec}_b(L)$. Suppose finally that
\[
B_{y_0, \he}(L):  t^\delta H^m_b( \RR^+\times F; t^{-1} dt dz) \longrightarrow t^\delta L^2(\RR^+ \times F; t^{-1} dt dz)
\]
is invertible for every $(y_0, \he)$. Then both \eqref{Lsob} and \eqref{Lhold} are Fredholm mappings. If
we only know that $B(L)$ is injective for all $(y_0, \he)$, then \eqref{Lsob} and \eqref{Lhold} are
semi-Fredholm and essentially injective; if $B(L)$ is surjective for every $(y_0, \he)$, then \eqref{Lsob} and \eqref{Lhold}
are semi-Fredholm and essentially surjective. 
\label{semi-fred}
\end{prop}

\noindent{\bf Normalizations and conventions.} We first rewrite Assumption~\ref{assump-uniq-cont} in the
following form: 
\begin{assumption}
There exists values $\ld < \ud$, $\ld, \ud \notin \mbox{Spec}_b(L)$, such that, for every $(y_0, \he)$,
\[
B_{y_0, \he}(L): t^{\ld}H^m_b \longrightarrow t^{\ld} L^2
\]
is surjective, and
\[
B_{y_0, \he}(L):  t^{\ud}H^m_b \longrightarrow t^{\ud}L^2
\]
is injective.
\label{defdlb}
\end{assumption}
\begin{remark}
It is enough to assume that an `injectivity weight' $\ud$ exists for both $B(L)$ and its adjoint $B(L)^*$ (taken with
respect to any fixed measure of the form $t^\gamma dt dz$). This holds simply because injectivity of $B(L)^*$ 
on some $t^\delta L^2$ is equivalent to surjectivity of $B(L)$ on another space $t^{\delta^*}L^2$, where $\delta^*$
is determined by $\delta$ and $\gamma$. 
\end{remark}
Based on this, we see that Assumption~\ref{defdlb} will hold if both $B(L)$ and $B(L)^*$ satisfy the more basic
\begin{assumption}[Unique continuation property]
Any solution $u$ to $B(L) u = 0$ which vanishes to infinite order at $t = 0$ and which has subexponential growth as $t \to \infty$
is the trivial solution $u \equiv 0$. 
\label{uniq-cont}
\end{assumption}
That this should always be true is quite believable, but has not been proved in general. It is known to hold in
the special case where $L$ is second order with diagonal principal part and $\dim F = 0$, see \cite{Ma-UC}.

\begin{remark}
Another observation which simplifies notation below is that the precise choice of measure $t^\gamma dt dz$ for
$B(L)$, or $x^\delta dx dy dz$ for $L$, (or other measures which differ from these by a smooth function $J$
which is uniformly bounded above and away from $0$) is irrelevant for these various mapping and regularity
properties. Obviously, the values of $\gamma$ and $J$ enter into the precise computations of adjoints, normalization
of weight parameters, etc., but do not in any way effect the nature of the any of the results below. 

Thus we always assume that we are working with respect to the measure $dt dz$, or $dx dy dz$. 
We also fix the two values $\ld$ and $\ud$ (and this choice of fixed measures) henceforth for the rest of the paper. 
\label{fixmeas}
\end{remark}

One final remark: as noted earlier, we really only need to assume constancy of indicial roots with real part in the interval 
$[\ld, \ud]$, though in that more general case, one has slightly weaker regularity statements (conormality 
rather than complete polyhomogeneity). 

\medskip

\noindent{\bf Generalized inverses}
Assume that $L\in  \textup{Diff}^m_e(M)$ is elliptic and satisfies Assumptions~\ref{assump-const-ind-rts}
and \ref{uniq-cont}, and that $\ld$ and $\ud$ have been chosen as above.  By Proposition~\ref{semi-fred}, the mapping
\eqref{Lsob} is semi-Fredholm whenever $\delta \geq \ud$ or $\delta \leq \ld$, and in either case, $\delta \notin \mbox{Spec}_b(L)$;
in other words, this mapping has closed range, and either finite dimensional nullspace or finite dimensional cokernel, respectively.
General functional analysis then gives, for each $s \in \RR$, the existence of a generalized inverse $G$ for \eqref{Lsob}, which
is to say, there exists a bounded map
\begin{equation}
G:x^\delta H^s_e(M) \to x^\delta H^{s+m}_e(M)
\end{equation}
which satisfies $GL = I - P_1$, $LG = I - P_2$ where $P_1$ and $P_2$ are the orthogonal projectors onto the
nullspace of $L$ and orthogonal complement of the range of $L$, respectively.  By the simplest form of elliptic 
regularity in the edge setting, we obtain that $P_1$ and $P_2$ are both smoothing in the sense that
$P_1: x^\delta H^{s+m}_e \to x^\delta H^{t+m}_e$ and $P_2: x^\delta H^s_e \to x^\delta H^t_e$ are
bounded for any $t > s$. 

Much more is true, and one of the strengths of the pseudodifferential edge theory is that it allows one
to give a fairly explicit description of the Schwartz kernels of these operators.  Fix $\delta$ as above, and
set $s=0$ (to normalize the choice of projectors). Then Theorem 6.1 in \cite{M} asserts that $G$, $P_1$ and $P_2$ 
are all pseudodifferential edge operators. When $\delta > \ud$, then $P_1$ has finite rank and maps into the space 
of polyhomogeneous functions, while when $\delta < \ld$, then $P_2$ has finite rank and maps into the space of 
polyhomogeneous functions.  We shall recall the definitions of these spaces of pseudodifferential operators in \S 4,
but for now point out that this description of their Schwartz kernels has a number of important ramifications.  
For example, once one establishes a general boundedness
theorem for pseudodifferential edge operators on weighted edge H\"older spaces, then it is an immediate
consequence of this Sobolev semi-Fredholmness that one can then deduce that for this same value of $\delta$,
the mapping \eqref{Lhold} is also semi-Fredholm, and that $P_1$ and $P_2$ are the appropriate projectors
in that case too.  Indeed, the equations $GL = I - P_1$ and $LG = I - P_2$ still hold, and all operators are
bounded on the appropriate spaces. Note in particular that if $\delta < \ld$, for example, then the nullspace 
of \eqref{Lhold} is infinite dimensional and in this case it does not follow from general theory that this 
nullspace is complemented in $x^\delta \Lambda_e^{m+k}$. Nonetheless, since the infinite rank projectors 
$P_1$ and $I - P_1$ are bounded, we see that this nullspace has a complement, as claimed.

%%%%%%%%%%%%%%%%%%%%
\section{Outline and statement of the main result}\label{overview}
%%%%%%%%%%%%%%%%%%%%
We are now in a position to provide a more careful statement of our main results and to
sketch the arguments to prove them. 

There are several closely related conceptual frameworks for studying elliptic boundary problems;
the one we follow here is very close to the one developed by Boutet de Monvel \cite{BdM}, and used
in many other places since, including by Schulze \cite{S1, S2} for edge operators. This theory is 
centered around the idea of extending the use of `interior' pseudodifferential edge operators
by introducing the associated spaces of trace and Poisson operators, as well as boundary 
operators along the edge $B$. 

What distinguishes our approach here is the focus on the geometric structures of the
Schwartz kernels of these various types of operators.  As in \cite{M}, any one of
these operators has a Schwartz kernel which is a polyhomogeneous distribution
on a certain blown up space. Section~\ref{micro} describes all of this more carefully.
Amongst the tasks we must face is to show that the composition of an interior
edge operator and a trace operator (interior to boundary) is again a trace operator,
and similarly the composition of a Poisson operator (boundary to interior) with
an interior edge operator is again of Poisson type. These composition formul\ae\ 
are perhaps the most technically demanding part of this presentation. 

The operators which arise in these elliptic boundary problems are of a somewhat
more special type, which we call representable. This is described in \S \ref{representable},
where we introduce these subclasses of interior, trace and Poisson edge operators 
and examine their normal operators. 

Following these more general `structural' definitions and results, we turn to the analysis 
specific to elliptic differential edge operators.  In the steps below, we first define each object 
at the level of Bessel operators, where the issues are typically finite dimensional.  We then
rescale and take inverse Fourier transforms and obtain the corresponding objects at the
level of normal operators. Although everything becomes infinite dimensional, it is still
completely equivalent to the finite dimensional problem. The last step is to extend each
object from the normal operator level to that of the actual operators, and this is where
the special class of representable operators becomes important.  

The starting point is to identify the spaces on which the boundary trace map is well defined. 
We define
\begin{equation}
\begin{split}
\calH^{B(L)}_{\ud,\ld} & = \{u \in t^{\ld} H^m(\RR^+ \times F_{y_0}; dt\, dz) \mid B_{y_0,\he}(L) u \in t^{\ud} L^2\} \\
\calH^{N(L)}_{\ud,\ld} & = \{u \in s^{\ld} H^m(\RR^+ \times \RR^b \times F_{y_0}; ds\, dY \, dz) \mid N_{y_0}(L) u \in s^{\ud} L^2\} \\
\calH^{L}_{\ud,\ld} & = \{u \in x^{\ld} H^m(M; dx \, dy\, dz) \mid L u \in x^{\ud} L^2 \},
\end{split}
\label{defHsp}
\end{equation}
where by implication the second inclusion is supposed to hold for all $y_0, \he$.  For simplicity we often denote
these simply as $\calH^B$, $\calH^N$ and $\calH$, omitting the subscript $\ud,\ld$. These are Hilbert spaces with respect 
to the norms
\begin{equation}
\begin{split}
||u||_{\calH^B} & = || u||_{t^{\ld}L^2} + || B(L)u ||_{t^{\ud}L^2} \\ 
||u||_{\calH^N} & = || u||_{s^{\ld}L^2} + || N(L)u ||_{s^{\ud}L^2} \\ 
||u||_{\calH} & = || u||_{x^{\ld}L^2} + ||L u ||_{x^{\ud}L^2}.
\end{split}
\end{equation}

In \S \ref{trace-pot} we construct successively the trace and Poisson operators associated to an elliptic edge operator $L$
by first constructing the corresponding operators for $B(L)$ and $N(L)$.  Since $B(L)$ is Fredholm at all nonindicial weights, 
most of the considerations for it are finite dimensional and we may formulate the analogue of the Calderon, or
Lopatinski-Schapiro conditions directly.  The starting point is that $\calH^B$ is the natural domain for the boundary trace
map for $B(L)$, and in fact for each $y_0 \in B$,
\[
\Tr_{B(L)}: \calH^B  \longrightarrow \calE_{y_0} := \bigoplus \calE _{y_0} (L, \zeta_j), 
\]
where $\calE_{y_0}(L, \zeta_j)$ is the fibre of the trace bundle \eqref{trace-bundle} associated to the indicial root $\zeta_j$ at $y_0$ 
and the direct sum is over all indicial roots with imaginary part in the interval $(\ld-1/2, \ud-1/2)$.  The
corresponding trace map for the normal operator $N(L)$ is obtained by rescaling and taking the inverse Fourier transform, and
\[
\Tr_{N(L)}:  \calH^N  \longrightarrow \bigoplus H^{-(\Im (\zeta_j) -\ld +1/2)}(\RR^b; \calE_{y_0}). 
\]
The trace map for $L$ itself is bounded as a map
\[
\Tr_{L}: \calH \longrightarrow \bigoplus H^{-(\Im (\zeta_j) -\ld +1/2)}(B; \calE). 
\]
In a similar way, we construct the Poisson edge operators $P_{B(L)}$, $P_{N(L)}$ and 
\[
P_L: \bigoplus H^{-(\Im (\zeta_j) -\ld +1/2)}(B, \calE(\zeta_j)) \longrightarrow \ker L \cap x^{\ld}H^{\infty}_e(M). 
\]
By construction, $P_L \circ \Tr_L$ is the identity on $\ker L \cap \calH$.  The Calderon subspaces
\begin{multline*}
\calC_{B(L)} = \Tr_{B(L)} ( \ker B(L) \cap \calH^B), \quad \calC_{N(L)} = \Tr_{N(L)} ( \ker N(L) \cap \calH^N), \quad \mbox{and}\\
\hfill \calC_{L} = \Tr_{L} ( \ker L \cap \calH) \hfill
\end{multline*}
are of fundamental importance. For $B(L)$ this subspace depends smoothly on $(y_0,\he)$, and for $N(L)$ it
depends smoothly on $y_0$. 

Let us now explain how to formulate a boundary problem for the edge operator $L$.  Fix a vector bundle $W$ over $B$
and a pseudodifferential operator 
\[
Q:\calC^\infty (B, \calE) \to \calC^\infty (B,W).
\]
For many operators of interest, $W$ splits as a finite direct sum $\bigoplus W_k$, and of course $\calE$ also splits
into the summands corresponding to each indicial root, so $Q$ has a matrix form $(Q_{jk})$ where the different components
may have different orders. 

\begin{defn}
With all notation as above, an edge boundary value problem $(L,Q)$ is a system 
\begin{align*}
Lu & = f\in x^{\ud}L^2(M), \ u \in \calH_{\ld,\ud} \subset x^{\ld}H^m_e(M), \\
Q (\Tr_L u) & =\phi \in \bigoplus\limits_{k=1}^M H^{\ld -d_k-1/2}(B,W_k). 
\end{align*}
\end{defn}
As in the classical theory on a manifold with boundary, the determinantion of whether this problem is Fredholm is formulated
using the (left or right) invertibility of the principal symbol of the boundary conditions restricted to the Calderon subspace: 
\begin{defn}
The boundary conditions $Q$ of an edge boundary value problem $(L,Q)$ are 
\begin{enumerate}
\item right-elliptic if $\sigma(Q)(y_0,\hat\eta)\restriction \calC_{B(L)(y_0,\he)}: \calC_{B(L)(y_0,\he)} \to \pi^* W_{y_0}$ 
is surjective,
\item left-elliptic if $\sigma(Q)(y_0,\hat{\eta})\restriction \calC_{B(L)(y_0,\he)}: \calC_{B(L)(y_0,\he)} \to \pi^* W_{y_0}$ 
is injective, and
\item elliptic if $\sigma(Q)(y_0,\hat{\eta})\restriction \calC_{B(L)(y_0,\he)}:\calC_{B(L)(y_0,\he)} \to \pi^* W_{y_0}$ 
is an isomorphism
\end{enumerate}
for all $(y_0,\hat\eta) \in S^*B$, where $\pi: S^*B \to B$ is the standard projection.
\label{typesofbcs}
\end{defn}

The final section, \S \ref{fredholm}, assembles the various types of operators considered earlier to construct parametrices
in each of these three cases. Our main result is the 
\begin{thm}
Let $(L,Q)$ be right-elliptic.  Let $G$ be the generalized inverse for $L$ on $x^{\ld}L^2$. Then 
\[
(L,Q): (\calH,\|\cdot \|_\calH)\to x^{\ud}L^2(M)\oplus \left(\bigoplus\limits_{k=1}^MH^{\ld-d_k-1/2}(B,W_k)\right),
\]
is semi-Fredholm with right parametrix
\begin{align*}
\mathcal{G}(f,\phi)= Gf + P_L[K(\phi - Q(\Tr_{L}Gf))].
\end{align*}
In particular, $(L,Q)$ has closed range of finite codimension.
\end{thm}

\begin{thm}
Let $(L,Q)$ be left-elliptic. Then 
\[
(L,Q): (\calH,\|\cdot \|_\calH)\to x^{\ud}L^2(M)\oplus \left(\bigoplus\limits_{k=1}^MH^{\ld-d_k-1/2}(B,W_k)\right),
\]
is semi-Fredholm with left parametrix
\begin{align*}
\mathcal{G}(f,\phi)= Gf + P_L[K(\phi - Q(\Tr_{L}\, Gf))].
\end{align*}
In particular, $(L,Q)$ has a finite-dimensional kernel.
\end{thm}

These results together prove that an elliptic edge boundary problem gives a Fredholm mapping. 

%%%%%%%%%%%%%%%%%%%%
\section{Interior, trace and Poisson edge operators}\label{micro}
%%%%%%%%%%%%%%%%%%%%
In this section we recall the space of pseudodifferential edge operators and introduce the corresponding spaces 
of trace and Poisson operators. As explained earlier, our focus is on the Schwartz kernels of these operators, in
particular their structure as polyhomogeneous distributions.  We keep the notation of the preceding sections.

The definitions below are phrased in the language of manifolds with corners and various spaces of conormal or
polyhomogeneous functions on them, so we review some of this now. A manifold with corners is a space locally 
diffeomorphically modelled on neighbourhoods in the standard orthant $(\RR^+)^\ell \times \RR^{n-\ell}$. A standing 
assumption is that every boundary face of a manifold with corners is embedded. 
This implies, in particular, that if $H$ is a boundary hypersurface, then there is a globally defined boundary defining 
function $\rho_H$ which vanishes precisely on $H$ and is strictly positive everywhere else, and is such that $d\rho_H \neq 0$ at $H$.  

The most useful and natural classes of `smooth' functions on a manifold with corners $\mathfrak{W}$ are the conormal and 
polyhomogeneous distributions. Let $\{(H_i,\rho_i)\}_{i=1}^N$ enumerate the boundary hypersurfaces and corresponding
defining functions of $\mathfrak{W}$. For any multi-index $b= (b_1,\ldots, b_N)\in \C^N$ set $\rho^b = \rho_1^{b_1} \ldots \rho_N^{b_N}$.  
Similarly, for $p = (p_1, \ldots, p_N) \in \NN_0^N$, we write $(\log \rho)^p = (\log \rho_1)^{p_1} \ldots (\log \rho_N)^{p_N}$. 
Finally, let $\calV_b(\mathfrak{W})$ be the space of all smooth vector fields on $\mathfrak{W}$ which are unconstrained
in the interior but which lie tangent to all boundary faces.
\begin{defn}\label{phg}
A distribution $u$ on $\mathfrak{W}$ is said to be conormal of order $b$ at the faces of $\mathfrak{W}$, 
written $u \in \scrA^b(\mathfrak{W})$, 
if $u\in \rho^b L^\infty(\mathfrak{W})$ for some $b\in \C^N$ and $V_1 \ldots V_\ell u \in \rho^b L^\infty(\mathfrak{W})$ 
for all $V_j \in \calV_b(\mathfrak{W})$ and for every $\ell \geq 0$. 

An index set $E$ is a collection of pairs $\{(\gamma,p)\} \subset \CC \times \NN_0\}$ satisfying the following hypotheses:
\begin{enumerate}
\item $\Re \gamma$ accumulates only at plus infinity, while the second index $p$ for a given $\gamma$ is bounded above 
by a constant depending on $\gamma$, i.e.\ $p \leq P_\gamma < \infty$;
\item If $(\gamma,p) \in E$, then $(\gamma+j,p') \in E_i$ for all $j \in \NN$ and $0 \leq p' \leq p$. 
\end{enumerate}
An index family $\calE = (E_1, \ldots, E_N)$ is an $N$-tuple of index sets associated to each of
the boundary hypersurfaces of $\mathfrak{W}$.  In the rest of this paper, we typically let $k$ stand for 
the simple index set $\{(k+\ell,0): \ell \in \NN_0\}$. 

A conormal distribution $u$ on $\mathfrak{W}$ is said to be polyhomogeneous with index family $\calE$, 
$u \in \scrA_{\phg}^\calE(\mathfrak{W})$, if $u\in \scrA^*$, and if in addition, near each $H_i$, 
\[
u \sim \sum_{(\gamma,p) \in E_i} a_{\gamma,p} \rho_i^{\gamma} (\log \rho_i)^p, \ \mbox{as} \ \rho_i\to 0,
\]
with coefficients $a_{\gamma,p}$ conormal on $H_i$, polyhomogeneous with index $E_j$ at any $H_i\cap H_j$. 
We also require that $u$ have product type expansions at all corners of $\mathfrak{W}$. 
\end{defn}

A $p$-submanifold in a manifold with corners $\mathfrak{W}$ is an embedded submanifold with the property that
if $p \in S$, then it is possible to choose coordinates $(x,y) \in (\RR^+)^k \times \RR^{n-k}$ for $\mathfrak{W}$ 
with $p = (0,0)$, and such that $S = \{(x,y): x'' = 0, y'' = 0\}$, where $x = (x',x'')$ and $y = (y', y'')$ are some
subdivisions of these sets of coordinates. In other words, $\mathfrak{W}$ has a product structure near $S$.
We may then define the new manifold with corners $[\mathfrak{W}, S]$ by blowing up $\mathfrak{W}$ around $S$.
This consists of taking the disjoint union $\mathfrak{W} \setminus S$ and the inward-pointing normal bundle
of $S$, and endowing this set with the structure of a smooth manifold with corners, with the unique minimal
differential structure so that smooth functions on $\mathfrak{W}$ and polar coordinates around $S$ all lift to
be smooth. This blown up space has a `front face', which is a new boundary hypersurface which projects
down to $S$ in the `blowdown'; it is the total space of a fibration over $S$ with fibre some spherical orthant.

%%%%%%%%%%%%%%
\subsection{Pseudodifferential edge operators}\label{edge-pseudos}
%%%%%%%%%%%%%%
Let $M^2_e$ denote the double edge space, which is obtained by blowing up the fibre diagonal of $(\del M)^2$ 
in the product $M^2$, $M^2_e = [ M^2; \mathrm{fdiag}]$.  In standard adapted local coordinates $(x,y,z)$ 
on $M$ near $\del M$, with $(\wx,\wy,\wz)$ a copy of these coordinates on the other factor of $M$ in $M^2$, 
the fibre diagonal $\mathrm{fdiag}$ is the submanifold $\{x = \wx = 0, y = \wy\}$; it is the total space of a fibration over
$\mbox{diag}\,(B \times B)$ with fibre $S^n_+ \times F \times F$. The space $M^2_e$ is a manifold with corners up to codimension 
three; there are three boundary hypersurfaces, denoted $\ff$ (the front face), $\lf$ (the left face) and $\rf$ (the right face). 
The front face is the one created by the blowup; it is the total space of a fibration over $\mathrm{fdiag}$ with each fibre a copy 
of the quarter-sphere $\{\omega = (\omega_0, \omega', \omega_n) \in S^n: \omega_0, \omega_n \geq 0\}$. 

It is often more convenient to use projective coordinates rather than polar coordinates. Thus away from $\rf$, we use 
\begin{equation}
\label{proj-coord}
s=\frac{x}{\wx}, \, Y=\frac{y-\wy}{\wx}, \, z, \, \wx, \, \wy, \, \wz,
\end{equation}
where $\wx$ and $s$ are defining functions of $\ff$ and $\rf$, respectively. Note that in these coordinates, $\ff$
is the face where $\wx = 0$. There are analogous coordinates  valid away from $\rf$, obtained by interchanging the roles of $x$ and $\wx$. 

Figure 1 illustrates $M^2_e$

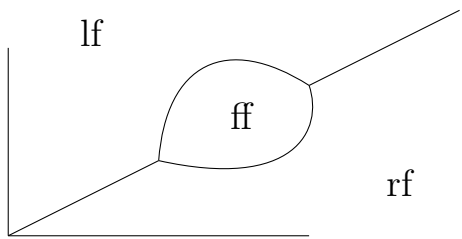
\begin{figure}[h]
\begin{center}
\begin{tikzpicture}
\draw (0,0) -- (2,1);
\draw (4,2) -- (6,3);
\draw (0,0) -- (4,0);
\draw (0,0) -- (0, 2.5);
\draw (2,1) .. controls (2.1,2.3) and (2.9,2.7) .. (4,2);
\draw (2,1) .. controls (3.8,0.6) and (4.2,1.4) .. (4,2);
\node at (3.1,1.6) {\large{ff}};
\node at (1.1,2.7) {\large{lf}};
\node at (5.2,0.7) {\large{rf}};
\end{tikzpicture}
\end{center}
\label{figure-edge}
\caption{The edge double space $M^2_e$.}
\end{figure}
This space has a distinguished submanifold, the edge diagonal $\mathrm{diag}_e$, which is the lift of 
the diagonal to $M^2_e$. (Strictly speaking, it is the closure of the lift of the interior of the diagonal.) 

A linear operator $A$ on $M$ is called a pseudodifferential edge operator of order $m$ and with index family 
$\calE$, $A \in \Psi_e^{m, \calE}(M)$, if the lift of its Schwarz kernel $K_A$ to $M^2_e$ is polyhomogeneous 
distribution on this space, where the index sets $\calE = (E_{\ff}, E_{\lf}, E_{\rf})$ describe the expansions at the three faces. The superscript $-\infty$ indicates
the pseudodifferential order, hence the lifted Schwartz kernel is smooth along $\mbox{diag}_e$.  The full space
of pseudodifferential edge operators, $\Psi^{*,\calE}_e(M)$, consists of the space of sums $A + B$ where 
$A$ is an operator of order $-\infty$ as above, and where the lift of the Schwartz kernel of $B$ to $M^2_e$ 
is supported near $\mbox{diag}_e$, has a classical conormal singularity along that submanifold, and is smoothly
extendible (after factoring out a certain singular density) across $\ff$.  To understand the singular density
here, note that the identity operator has Schwartz kernel which lifts as 
\[
\delta(x-\wx) \delta(y-\wy) \delta(z-\wz) d\wx d\wy d\wz = \delta(s-1) \delta(Y) \delta(z-\wz) \, \wx^{-b-1} d\wx d\wy d\wz.
\]
This is smoothly extendible across the front face, which in these projective coordinates is where $\wx = 0$, 
provided we factor out the final singular measure.   In the language above, $\mbox{Id} \in \Psi^{0, \varnothing}_e(M)$. 

There is a distinguished subalgebra $\Psi_e^{*}(M)$, called the small calculus, which consists of operators
which vanish to infinite order at the left and right faces, $E_{\lf} = E_{\rf} = \varnothing$, and with
$E_{\ff} = 0$.   The residual calculus $\Psi_e^{-\infty, 0, E_{\lf}, E_{\rf}}(M)$ consists of operators with no singularity
along the lifted diagonal and with standard index set $0$ at the front face. 

Many details have been suppressed here, and we refer to \cite{M} where all of this is described more carefully. 

%%%%%%%%%%%%%%%%
\subsection{Edge trace operators}
%%%%%%%%%%%%%%%%
Whereas the edge operators introduced in the previous subsection map functions on $M$ to functions on $M$, the
other two classes of operators we consider map functions on $M$ to functions on $\del M$ (these are the edge
trace operators) or functions on $\del M$ to functions on $M$ (these are the edge Poisson operators.  We now
describe the former of these. 

An edge trace operator $T$ is again described in terms of the lifting properties of its Schwartz kernel.
Initially this Schwartz kernel is a distribution on $\del M \times M$; this space has the same distinguished
submanifold as before, namely the fibre diagonal of $(\del M)^2$, $\mathrm{fdiag} = \{\wx = 0, y = \wy\}$.
We define the edge trace double space 
\[
T^2_e = [\del M \times M; \mathrm{fdiag}];
\]
note that this is nothing other than the right face $\rf$ of $M^2_e$. It has two boundary hypersurfaces, 
the new front face of which, still denoted here by $\ff$, is simply one boundary face of the front face
of $M^2_e$, and hence a bundle of hemispheres $S^{n-1}_+$ over $\mathrm{fdiag}$. The lift of the original
face here is denoted $\of$, and still called the original face. 

We can use the same projective coordinates as before, namely $(Y, z, \wx, \wy, \wz)$ with $Y = (y-\wy)/\wx$. 
Figure 2 illustrates this space. 

\begin{figure}[h]
\begin{center}
\begin{tikzpicture}
\draw (0,0) -- (2,1);
\draw (4,2) -- (5.5,2.75);
\draw (0,0) -- (0, 2.5);
\draw (2,1) .. controls (2.1,2.3) and (3.2,3) .. (4,2);
\node at (3.1,1.8) {\large{ff}};
\node at (6.1,3.1) {\large{rf}};
\end{tikzpicture}
\end{center}
\caption{The trace blowup $T^2_e$.}
\label{fig-trace}
\end{figure}
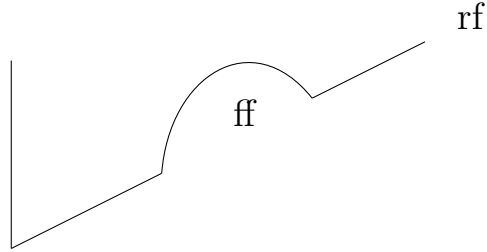

\begin{defn}\label{trace-op}
The space $\Psi^{k, F_{\rf}}_e(M)$ of trace operators of order $k\in \N_0$ 
is the space of all operators $T$ with Schwartz kernels $K_T$ which are pushforwards 
from polyhomogeneous conormal distributions $\kappa_T$ on the trace blowup space $T^2_e$ which
have index set $F_{\of}$ at the original face, and index set $F_{\ff}= (-1-b + k)+\N_0$ at the front face.
\end{defn}

%%%%%%%%%%%%%%%
\subsection{Edge Poisson operators}
%%%%%%%%%%%%%%%

The last class of operators we define are those which act from functions on $\del M$ to functions on $M$.
The ones amongst these in which we are particularly interested are analogues of the classical Poisson operators,
and hence take functions on the boundary to functions in the interior which are solutions of an elliptic edge operator $L$. 
However, it is advantageous to consider the full class of all operators with the relevant structure. 

The Schwartz kernel of an edge Poisson operator $P$ is a distribution on $M \times \del M$, and as usual,
we consider distributions which lift to be polyhomogeneous on the edge Poisson double space $P^2_e$,
obtained from $M \times \del M$ by blowing up the same fibre diagonal $\mathrm{fdiag}$.  The space
$P^2_e$ is naturally identified with the left face $\lf$ of $M^2_e$; it has two boundary hypersurfaces,
the front face $\ff$, which is `the other' boundary hypersurface of the front face of $M^2_e$, and the
original face $\of$. We often use projective coordinates $(x, z, Y, \wy, \wz)$ with $Y=(y-\wy)/x$. 
It is illustrated in Figure 3 (which is just the `transpose' of Figure 2). 

\begin{figure}[h]
\begin{center}
\begin{tikzpicture}
\draw (0,0) -- (2,1);
\draw (4,2) -- (5.5,2.75);
\draw (0,0) -- (4,0);
\draw (2,1) .. controls (4,0.6) and (4.3,1.7) .. (4,2);
\node at (3.3,1.5) {\large{ff}};
\node at (6,3) {\large{lf}};
\end{tikzpicture}
\end{center}
\caption{The Poisson double space $P_e$.}
\label{fig-pot}
\end{figure}
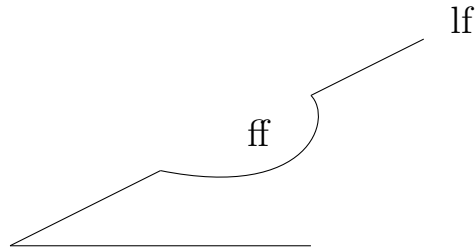

\begin{defn}\label{potential-op}
The space $\Psi^{ k, J_{\lf}}_e(P^2_e)$ of edge Poisson operators of order $k\in \N_0$
is the space of all operators $P$ with Schwartz kernels $K_P$ which are pushforwards 
from polyhomogeneous conormal distributions $\kappa_P$ on the edge Poisson double
space $P^2_e$, with index sets $J_{\of}$ at the original face, and $J_{\ff}= (-1-b+k)+\N_0$ at the front face.
\end{defn}

Comparing with the Boutet de Monvel calculus, one expects that we should also include operators mapping
functions on $\del M$ to functions on $\del M$. Indeed, a complete analogue of that calculus (as in the work 
of Schulze) would indeed include these, but this is not necessary for our purposes here.  Note that the operators
of this type we would need are not of any particularly standard type; their Schwartz kernels on $(\del M)^2$
should be conormal at the fibre diagonal $\mathrm{fdiag}$, rather than the diagonal of the boundary.
These are, in some sense, lifts of pseudodifferential operators from $B^2$ to $(\del M)^2$. 

%%%%%%%%%%%%%%%%
\subsection{Composition formul\ae}\label{triple}
%%%%%%%%%%%%%%%%
The key fact which makes the definitions above useful is that these classes of operators are closed
under composition. This statement must be qualified to account for two issues. The first is the trivial
observation that one can only compose operators of the appropriate types, e.g.\ $T \circ A$ is defined
if $A$ is an interior edge operator and $T$ is an edge trace operator, and similarly, $A \circ P$ is defined
if $P$ is an edge Poisson operator and $A$ an interior edge operator, but of course $P \circ T$ is not defined, etc.
More seriously, however, even when composing two interior edge operators, the composition may not be defined because 
of integrability issues. Thus if $A \in \Psi^{*,\calE}_e$ and $A' \in \Psi^{*,\calE'}_e$, then $A \circ A'$ is defined only
if $E_{\rf} + E'_{\lf} > -1$ (this lower bound depends on the choice of reference measure). The full composition 
theorem for interior edge operators is proved in \cite{M}, and we prove here the analogous results 
for compositions involving edge Poisson and trace operators.  The main point in all of this is the more subtle fact 
that if two operators have Schwartz kernels which lift to be polyhomogeneous on the appropriate blown-up space, then the
same is true for the composition.  This can be verified `by hand', breaking up the 
regions of integration into different neighbourhoods and using projective coordinate systems
to check the polyhomogeneity of these localized integrals. There is a much more elegant
and conceptual way, due to Melrose, and employed in \cite{M} (and many
other places), using the `pushforward theorem'. This states that under appropriate conditions on 
a map $f: X \to X'$ between two manifolds with corners, the pushforward of a polyhomogeneous 
distribution is polyhomogeneous.  We review this result now and apply it to state the composition formul\ae. 

First introduce some terminology.  Let $X$ and $X'$ be two compact manifolds with corners, and $f: X \to X'$ 
a smooth map.  Let $\{H_i\}$ and $\{H_j'\}$ be enumerations of the codimension one boundary faces of $X$ and $X'$,
respectively, and let $\rho_i$, $\rho_j'$ be global defining functions for $H_i$, resp.\ $H_j'$. We say that the map $f$
is a $b$-map if 
\[
f^* \rho_i' = A_{ij} \prod_i \rho_j^{e(i,j)}, \quad 0 < A_{ij}  \in \calC^\infty(X),\ e(i,j) \in \mathbb{N} \cup \{0\};
\]
in other words, $f^* \rho_j'$ vanishes to constant order along each boundary face of $X$.  In particular, this
means that if $f(H_i) \cap H_j' \neq \varnothing$, then $f(H_i) \subset H_j'$, and the order of vanishing of $f$ in the
direction normal to $H_i$ is constant along the entire face. 

Next, $f$ is called a $b$-submersion if $f_*$ induces a surjective map between the $b$-tangent bundles of $X$ and $X'$. (The 
$b$-tangent space at a point $p$ of $\del X$ on a codimension  $k$ corner is spanned locally by  the
sections $x_1 \del_{x_1}, \ldots, x_k \del_{x_k}, \del_{y_j}$, where $x_1, \ldots, x_k$ are the defining functions for the faces
meeting at $p$ and the $y_j$ are local coordinates on the corner through $p$.)  Finally, if we require that $f$ is not
only a $b$-submersion, but that in addition, for each $j$ there is at most one $i$ such that $e(i,j) \neq 0$ 
(this condition simply means that each hypersurface face $H_i$ in $X$ gets mapped into \emph{at most one} 
$H_j'$ in $X'$, or in other words, no hypersurface in $X$ gets mapped to a corner in $X'$), then $f$ is called a $b$-fibration. 

Let $\nu_0$ be any smooth density on $X$ which is everywhere nonvanishing and smooth up to all boundary faces of $X$. 
A smooth $b$-density $\nu_b$ is, by definition, any density of the form $\nu_b =\nu_0 (\Pi \rho_i)^{-1}$.  
Fix smooth nonvanishing $b$-densities $\nu_b$ on $X$ and $\nu_b'$ on $X'$. 
\begin{prop}[The Pushforward Theorem (Melrose)]
Let $u$ be a polyhomogeneous function on $X$ with index set $E_j$ at the face $H_j$, for all $j$.  Suppose that if
$e(i,j) = 0$ for all $i$, i.e.\ $H_j$ is mapped to the interior of $X'$, then $\mbox{Re}\, z > 0$ for all $(z,p) \in E_j$. 
In this case, the pushforward $f_* (u \nu_b)$ is well-defined and equals $h \nu_b'$ where $h$ is polyhomogeneous 
on $X'$ and has an index family $f_b(\calE)$ given by an explicit formula in terms of the index family $\calE$ for $X$.
\end{prop}
We do not state the formula for the index set of the pushforward in generality, but give an informal description sufficient 
for the present situation.  If $H_{j_1}$ and $H_{j_2}$ are both mapped to a face $H_i'$, and if $H_{j_1} \cap H_{j_2} = \varnothing$, 
then the pushforward has index set  $E_{j_1} + E_{j_2}$ at $H_i'$. If they do intersect, then the contribution is the extended 
union $H_{j_1} \overline{\cup} H_{j_2}$. For any two index sets $E,E'$ their \emph{extended union} $E\overline{\cup}E'$ 
is defined by
\begin{align}
\label{extended}
E \overline{\cup} E' = E \cup E' \cup \{((z, p + q + 1): \exists \, (z,p) \in E,\ 
\mbox{and}\  (z,q) \in E' \}.
\end{align}

After these generalities, we can now state the composition results between interior and Poisson operators and 
between Poisson and trace operators. Note that the composition formula between trace and interior operators 
is the adjoint of the interior-Poisson composition, so we do not state it separately. 

\subsubsection{Interior $\circ$ Poisson} 
Let $G$ be an interior edge operator and $P$ a Poisson edge operator and consider the (only possible) composition $A=G \circ P$.
To show that this is again a Poisson edge operator, we must verify that the Schwartz kernel of this composition 
lifts to be polyhomogeneous on $P^2_e$ and has the stated index sets. This is accomplished by 
constructing the interior-Poisson triple space $M^3_{i-p}$, obtained by a sequence of blowups from $M \times M \times \del M$.
Recall the fibre diagonal $\fdiag$ which is blown up in the definitions of the interior edge and Poisson operators.
Here, using local coordinates $(x,y,z)$, $(x', y', z')$ and $(y'',z'')$ in the three factors, $\fdiag_{g} := \{ x = x'  = 0, 
y = y'\}$ is the fibre diagonal that needs to blown for polyhomogeneity of $G$, $\fdiag_{p} := \{x' = 0, y' = y''\}$
is the fibre diagonal that needs to blown for polyhomogeneity of $P$, and finally $\fdiag_{a} := \{x = 0, y = y''\}$
is the fibre diagonal that needs to blown for polyhomogeneity of  $A$. All the three submanifolds intersect at
$\fdiag_{0} := \{x=x'=x''=0, y=y'=y''\}$. We define the triple space by 
$$
M^3_{i-p}:=[[M \times M \times \del M; \fdiag_{0} ]; \fdiag_{g}, \fdiag_{p}, \fdiag_{a}].
$$
Then there exist natural projections 
\begin{align*}
&\pi_p: M^3_{i-p} \to P^2_e \times M_{(x,y,z)} \to P^2_e, \\
&\pi_a: M^3_{i-p} \to P^2_e \times M_{(x',y',z')} \to P^2_e, \\
&\pi_g: M^3_{i-p} \to M^2_e \times \partial M_{(y'',z'')} \to M^2_e.
\end{align*}
The maps $\pi_*$ are $b$-fibrations by construction of the triple space. 
The triple space is also equipped with the natural blowdown map $\beta_3: M^3_{i-p} \to M \times M \times \del M$.
We also consider the natural blowdown maps $\beta_2: M^2_e \to M^2$ and
$\beta_1: P^2_e \to M \times \partial M$.

The Schwartz kernel $K_G$ of an interior edge operator $G\in \Psi^{-\infty, \gamma, E_{\lf}, E_{\rf}}_e(M^2_e)$ lifts 
to a polyhomogeneous conormal distribution $\kappa_G= \beta_2^* K_G$ on $M^2_e$ of leading order $(-1-b+\gamma)$
at the front face. The Schwartz kernel $K_P$ of an edge Poisson operator  $P\in \Psi^{-\infty, \rho, J_{\rf}}_e(P^2_e)$ lifts 
to a polyhomogeneous conormal distribution $\kappa_P= \beta_1^* K_P$ on $P^2_e$ of leading order $(-1-b+\rho)$
at the front face.  The kernel of the composition $A=G \circ P$ 
can be expressed using pullbacks and pushforwards as 
\[
\kappa_A:=\beta_1^*(K_{A})= (\pi_a)_*[\pi_g^*\kappa_G \cdot \pi_p^*\kappa_P].
\]
Applying the pushforward theorem we obtain the
\begin{thm}\label{triple-ep}
If $\Re E_{\lf} + \Re J_\rf> -1$ then 
\[
\Psi^{-\infty, \gamma, E_{\rf}, E_{\lf}}_e(M^2_e) \circ \Psi^{-\infty, \rho, J_\rf}_e(P^2_e) \subset \Psi^{-\infty, \gamma+\rho, E_{\rf}}_e(P^2_e).
\]
\end{thm}
\begin{proof}
In view of the pushforward theorem, it remains to identify the leading order behaviour 
at the various boundary faces. Denote the boundary defining functions of the boundary faces introduced by blowing up 
$\fdiag_*$ by $\rho_*$, where $*\in \{0, g,p,a\}$. We write $\rho_\ff$ for the front 
face defining functions in the double spaces $M^2_e$ and $P^2_e$. The defining functions of the boundary faces 
$\{x=0\}$ and $\{x'=0\}$ in $M^3_{i-p}$ are denoted by $\rho_r$ and $\rho_l$, respectively.
The defining functions of the boundary faces 
$\{x=0\}$ and $\{x'=0\}$ in either $M^2_e$ or $P^2_e$ are denoted by $\rho_\rf$ and $\rho_\lf$, respectively.
We then obtain
\begin{align*}
&\pi_g^*(\rho_\ff) = \rho_0 \rho_g, \ \pi_g^*(\rho_{\rf, \lf}) = \rho_{r,l}, \\
&\pi_p^*(\rho_\ff) = \rho_0 \rho_p, \ \pi_g^*(\rho_{\lf}) = \rho_{l}, \\
&\pi_a^*(\rho_\ff) = \rho_0 \rho_a, \ \pi_g^*(\rho_{\rf}) = \rho_{r}. \\
\end{align*}
We denote by $\nu_3$ a $b$-volume on $M^3_{i-p}$ and by $\nu_1$ a $b$-volume on $P^2_e$.
We compute
\begin{align*}
&\beta_3^*(dx\, dy\, dz\, dx'\, dy'\, dz'\, dy''\, dz'') = \rho_0^{2+2b} (\rho_g \rho_p \rho_r \rho_l)^{1+b} \nu_3, \\
&\beta_1^*(dx\, dy\, dz\, dy''\, dz'') = \rho_\ff^{1+b} \rho_\rf^{1+b} \nu_1.
\end{align*}
The leading order behaviour of $\kappa_A$ at the various boundary faces of $P^2_e$ follows
now from the following computation
\begin{align*}
&\kappa_A \cdot \beta_1^*(dx\, dy\, dz\, dy''\, dz'') = 
\beta_1^*(K_{A} dx\, dy\, dz\, dy''\, dz'') \\ &= 
(\pi_a)_*[\pi_g^*\kappa_G \cdot \pi_p^*\kappa_P \cdot \beta_3^*(dx\, dy\, dz\, dx'\, dy'\, dz'\, dy''\, dz'')] \\
&= (\pi_a)_*[ \rho_0^{\gamma+\rho} \rho_g^{\gamma} \rho_p^{\rho} \rho_r^{1+b+E_{\rf}} \rho_l^{1+b+E_{\lf}+J_\rf} \nu_3] \\
&=\rho_\ff^{\gamma+\rho} \rho_\rf^{1+b+E_{\rf}} \nu_1 = \rho_\ff^{-1-b+ \gamma+\rho} \rho_\rf^{E_{\rf}} \beta_1^*(dx\, dy\, dz\, dy''\, dz''). 
\end{align*}
This proves the statement.
\end{proof}

\subsubsection{Poisson  $\circ$  trace}
Let $P$ be a Poisson and $T$ a trace edge operator, and consider the (only possible) composition $G=P \circ T$.
To show that this is again an interior edge operator, we must verify that the Schwartz kernel of this composition 
lifts to be polyhomogeneous conormal on $M^2_e$ and has the stated index sets. This is accomplished by 
constructing the Poisson-trace triple space $M^3_{p-t}$, obtained by a sequence of blowups from $M \times \del M \times M$.
Recall the fibre diagonal $\fdiag$ which is blown up in the definitions of the Poisson and the trace edge operators.
Here, using local coordinates $(x,y,z)$, $(y', z')$ and $(x'', y'',z'')$ in the three factors, $\fdiag_{p} := \{ x  = 0, 
y = y'\}$ is the fibre diagonal that needs to blown for polyhomogeneity of $P$, $\fdiag_{t} := \{x'' = 0, y' = y''\}$
is the fibre diagonal that needs to blown for polyhomogeneity of $T$, and finally $\fdiag_{g} := \{x = x''=0, y = y''\}$
is the fibre diagonal that needs to blown for polyhomogeneity of  $G$. All the three submanifolds intersect at
$\fdiag_{0} := \{x=x''=0, y=y'=y''\}$. We define the triple space by 
$$
M^3_{p-t}:=[[M\times \del M  \times M ; \fdiag_{0} ]; \fdiag_{g}, \fdiag_{p}, \fdiag_{t}].
$$
Then there exist natural projections 
\begin{align*}
&\pi_p: M^3_{i-p} \to P^2_e \times M_{(x'',y'',z'')} \to P^2_e, \\
&\pi_t: M^3_{i-p} \to T^2_e \times M_{(x,y,z)} \to T^2_e, \\
&\pi_g: M^3_{i-p} \to M^2_e \times \partial M_{(y',z')} \to M^2_e.
\end{align*}
The maps $\pi_*$ are $b$-fibrations by construction of the triple space. 
The triple space is also equipped with the natural blowdown map $\beta_3: M^3_{p-t} \to M\times \del M  \times M$.
We also consider the natural blowdown maps $\beta_g: M^2_e \to M^2$,
$\beta_p: P^2_e \to M \times \partial M$ and $\beta_t: T^2_e \to \partial M \times M$.

The Schwartz kernel $K_P$ of an edge Poisson operator $P\in \Psi^{-\infty, \rho, J_\rf}_e(P^2_e)$ lifts 
to a polyhomogeneous conormal distribution $\kappa_P= \beta_p^* K_P$ on $P^2_e$ of leading order $(-1- b+\rho)$
at the front face. The Schwartz kernel $K_T$ of an edge trace operator  $T\in \Psi^{-\infty, \tau, F_{\lf}}_e(T^2_e)$ lifts 
to a polyhomogeneous conormal distribution $\kappa_T= \beta_t^* K_T$ on $T^2_e$ of leading order $(-1-b+\tau)$
at the front face.  The kernel of the composition $G=P \circ T$ 
can be expressed using pullbacks and pushforwards as 
\[
\kappa_G:=\beta_2^*(K_G)= (\pi_g)_*[\pi_p^*\kappa_P \cdot \pi_t^*\kappa_T].
\]
Applying the pushforward theorem we obtain the
\begin{thm}\label{triple-pt}
\[
\Psi^{-\infty, \rho, J_{\rf}}_e(P^2_e) \circ \Psi^{-\infty, \tau , F_{\lf}}_e(T^2_e) \subset \Psi^{-\infty, \rho+\tau , J_{\rf},F_{\lf}}_e(M^2_e).
\]
\end{thm}
\begin{proof}
In view of the pushforward theorem, it remains to identify the leading order behaviour 
at the various boundary faces. Denote the boundary defining functions of the boundary faces introduced by blowing up 
$\fdiag_*$ by $\rho_*$, where $*\in \{0, g,p,t\}$. We write $\rho_\ff$ for the front 
face defining functions in the double spaces $M^2_e$ and $P^2_e, T^2_e$. The defining functions of the boundary faces 
$\{x=0\}$ and $\{x'=0\}$ in $M^3_{i-p}$ are denoted by $\rho_r$ and $\rho_l$, respectively.
The defining functions of the boundary faces 
$\{x=0\}$ and $\{x'=0\}$ in either $M^2_e$ or $P^2_e, T^2_e$ are denoted by $\rho_\rf$ and $\rho_\lf$, respectively.
We then obtain
\begin{align*}
&\pi_p^*(\rho_\ff) = \rho_0 \rho_p, \ \pi_p^*(\rho_{\rf}) = \rho_{r}, \\
&\pi_t^*(\rho_\ff) = \rho_0 \rho_t, \ \pi_t^*(\rho_{\lf}) = \rho_{l}, \\
&\pi_g^*(\rho_\ff) = \rho_0 \rho_g, \ \pi_g^*(\rho_{\rf,\lf}) = \rho_{r,l}. \\
\end{align*}
We denote by $\nu_3$ a $b$-volume on $M^3_{i-p}$ and by $\nu_2$ a $b$-volume on $M^2_e$.
We compute
\begin{align*}
&\beta_3^*(dx\, dy\, dz\, dx'\, dy'\, dx'\, dy''\, dz'') = \rho_0^{2+2b} (\rho_p \rho_t \rho_r \rho_l)^{1+b} \nu_3, \\
&\beta_g^*(dx\, dy\, dz\, dx''\, dy''\, dz'') = \rho_\ff^{1+b} \rho_\rf^{1+b} \rho_\lf^{1+b} \nu_2.
\end{align*}
The leading order behaviour of $\kappa_G$ at the various boundary faces of $M^2_e$ follows
now from the following computation
\begin{align*}
&\kappa_G \cdot  \beta_g^*(dx\, dy\, dz\, dx''\, dy''\, dz'') =
\beta_g^*(K_G dx\, dy\, dz\, dx''\, dy''\, dz'') \\ &= 
(\pi_g)_*[\pi_p^*\kappa_P \cdot \pi_t^*\kappa_T \cdot \beta_3^*(dx\, dy\, dz\, dy'\, dz'\, dx''\, dy''\, dz'')] \\
&= (\pi_g)_*[ \rho_0^{\rho+\tau} \rho_p^{\rho} \rho_t^{\tau} \rho_r^{1+b+ J_{\rf}} \rho_l^{1+b+F_{\lf}} \nu_3] \\
&=\rho_\ff^{\rho+\tau } \rho_\rf^{1+b+J_{\rf}} \rho_\lf^{1+b+F_{\rf}} \nu_2 = \rho_\ff^{-1-b+\rho+\tau} 
\rho_\rf^{J_{\rf}} \rho_\lf^{F_{\rf}} \beta_g^*(dx\, dy\, dz\, dx''\, dy''\, dz''). 
\end{align*}
This proves the statement.
\end{proof}

%%%%%%%%%%%%%%%%
\section{Representable subclass of edge, trace and Poisson operators}\label{representable}
%%%%%%%%%%%%%%%%  
Within the more general classes of residual edge, Poisson and trace operators there are subclasses of operators
for which  the restriction to the front face has a particular representation formula. We call these the subclasses
of representable operators.  We introduce these now, and then show how the composition formul\ae\ specialize
in this setting, proving in particular that the composition of representable operators is again representable.

%%%%%%%%%%%%%%%%%%%
\subsection{Representable residual edge operators}\label{edge-op}
%%%%%%%%%%%%%%%%%%%%
We may consider $\RR^+ \times F$ as a manifold with boundary with a trivial edge structure, where the base $B$ reduces
to a single point. The corresponding edge double-space thus corresponds to the somewhat simpler $b$-double
space, from \cite{Mel-APS}, \cite{M}, and is denoted $(\RR^+ \times F)^2_b$. This is a manifold with 
corners, with three boundary faces, the left, right and the front face.  Let $G_0(\wy, \he) \in \calA_{\phg}^{\calE'}((\RR^+\times F)^2_b)$,
where $\calE' = (E_{\ff}=\N_0, \calE = (E_{\lf}, E_{\rf}))$, and the lf, rf index sets are constant (at least in the critical range) when varying 
in smooth parameters $(\wy, \he) \in S^*B$. 

\begin{defn}\label{edge-Bessel}
Let $G_0(t,z,\wt,\wz;\widetilde{y}, \widehat{\eta}) \in \calA_{\phg}^{\calE'}$ as above.  Here $(\widetilde{y}, \widehat{\eta}) \in
S^*B$ are smooth parameters. We say that $G_0$ is edge Bessel operator, $G_0\in \Psi^{-\infty,\calE'}_b((\RR^+\times F)^2)$, 
if it satisfies the following two conditions:
\begin{enumerate}
\item $G_0(t,z,\wt,\wz;\widetilde{y}, \widehat{\eta})$ decreases rapidly as $t\to \infty$, locally uniformly in $(z,\wt,\wz)$, 
and as $\wt\to \infty$, locally uniformly in $(t,z,\wz)$; 
\item $G_0(t,z,\wt,\wz;\widetilde{y}, \widehat{\eta})$ admits a polyhomogeneous expansion as $t\to 0$, where the 
coefficient functions decrease rapidly as $\wt\to \infty$, uniformly in the other coordinates, and vice versa.
\end{enumerate}
\end{defn}

Following \cite[(5.18)]{M},  if $G_0$ is a edge Bessel operator and $k\in \N_0$, set
\begin{align}\label{N-edge}
N_k(G_0)= \int_{\R^b} e^{iY\eta} G_0(s|\eta|,z,  |\eta|,\wz; \wy, \widehat{\eta}) 
|\eta|^{-k+1} \dbar \eta.
\end{align}
The proof of \cite[Prop. 5.19]{M} shows that $N_k(G_0)$ is polyhomogeneous on the
front face $\ff$ of the edge double space $(\RR^+ \times \RR^b \times F)^2_e$. 

It will be convenient below to use the homogeneity rescaling 
\begin{align}
\kappa_\lambda u(x, \cdot ):=u(\lambda x, \cdot), \ x\in \R^+.
\end{align}
Consider a residual edge operator $\textup{Op}(G_0)\in \Psi^{-\infty, k, \calE}_e(M^2_e)$. 
By definition this acts on test functions $u$ supported near $\del M$ by 
\begin{multline*}
\left[\textup{Op}(G_0) u\right](x,y,z) \\ =\int e^{i(y-\widetilde{y})\eta} \kappa_{|\eta|}\circ 
G_0(x,z,\wx,\wz; y, \widehat{\eta}) \circ \kappa_{|\eta|}^{-1} u(\wx, \wy,\wz) 
|\eta|^{-k} \, \dbar \eta \, d\wx \, d\wy \, d\wz. 
\end{multline*}
The Schwartz kernel is thus 
\begin{equation}
\begin{split}
K_{\mathrm{Op}(G_0)}&(x,y,z, \wx, \wy, \wz) = \int_{\R^b} e^{i(y-\widetilde{y})\eta} 
G_0(x|\eta|,z,\wx |\eta|,\wz; y, \widehat{\eta}) |\eta|^{-k+1} \, \dbar \eta \\
& = \wx^{-1-b+k}  \int_{\R^b} e^{iY\eta} G_0(s|\eta|,z,  |\eta|,\wz; \wy+\wx Y, 
\widehat{\eta}) |\eta|^{-k+1} \, \dbar \eta \\
& =  \wx^{-1-b+k} N_k(G_0) + \mathcal{O}(\wx^{-b+k} ) .
\end{split}
\label{kernel-edge}
\end{equation}

The representable subcalculus of residual edge operators consists of those operators $G\in \Psi^{-\infty, k, \calE}_e(M^2_e)$, 
whose normal operator $N(G)$, defined as the restriction of $\rho_\ff^{1+b-k}\kappa_G$ to $\ff$, is given by 
$N_k(G_0)$ for some $G_0\in \Psi^{-\infty, \calE'}_b (\RR^+\times F)$.

Recall that if $L$ is an elliptic edge operator, then for any nonindicial weight $\delta \in (\ld, \ud)$,
there is a generalized inverse $G$ and projectors $P_1$ and $P_2$ onto the nullspace and cokernel.
By \cite[(4.22)]{M}, the lift of the Schwartz kernel of $P_1$ to $M^2_e$ is polyhomogeneous with index set 
\begin{equation}
\label{P-index}
\begin{split}
E_{\lf} &=\{(\zeta,p)\in \textup{Spec}_b(L) \mid \Im \zeta >\delta -1/2\}, \\
E_{\rf} &= \{(\zeta,p) \in \C\times \N_0 \mid (\zeta +2\delta, p)\in E_{\lf}\}, \
E_{\ff} = \N_0.
\end{split}
\end{equation} 
Furthermore, its normal operator $N(P_1)$ equals $N_0(P_{01})$  
where $P_{01}\in \Psi^{-\infty, \calE}_b((\RR^+\times F)^2)$ is the
projector onto the nullspace for the Bessel operator $B(L)$. Similarly, the lift of the Schwartz kernel of 
$P_2$ to $M^2_e$ is polyhomogeneous with index set 
\begin{equation}
\begin{split}
F_{\rf} &=\{(\zeta,p) \in \C\times \N_0 \mid (-\zeta-2\delta -1, p) \in \textup{Spec}_b(L), 
\Im \zeta >-\delta -1/2\}, \\
F_{\lf} &= \{(\zeta,p) \in \C\times \N_0 \mid (\zeta - 2\delta, p)\in F_{\rf}\}, \ F_{\ff} = \N_0,
\end{split}
\end{equation} 
and has normal operator $N(P_2) = N_0(P_{02})$, where $P_{02}\in \Psi^{-\infty, 
\mathcal{F}}_b((\R^+\times F)^2)$, $\calF=(F_{\ff}, F_{\lf},F_{\rf})$.  Note that 
if $\delta > \ud$ then $P_{01} = 0$ while if $\delta < \ld$ then $P_{02} = 0$.  
Finally, the lift of the Schwartz kernel of $G$ is polyhomogeneous on $M^2_e$ with index set 
\[
H_{\rf} =E_{\rf}\overline{\cup}F_{\rf}, \
H_{\lf} =E_{\lf}\overline{\cup}F_{\lf}, \
H_{\ff} = \N_0,
\]
This has normal operator $N(G) = N_0(G_0)$ for $G_0\in \Psi^{-\infty, \mathcal{H}}_b((\R^+\times F)^2)$, 
$\mathcal{H}=(H_{\ff},H_{\lf},H_{\rf})$.

%%%%%%%%%%%%%%%%%%%
\subsection{Representable trace operators}
%%%%%%%%%%%%%%%%%%%%

We next introduce the Bessel trace kernels.

\begin{defn}\label{trace-Bessel}
Let $T_0(\wt,z,\wz;\widetilde{y}, \widehat{\eta})$ be polyhomogeneous on $F^2\times \RR^+$, 
smooth in the interior, and varying smoothly in $(\widetilde{y}, \widehat{\eta})\in S^*B$, with 
index sets $\mathcal{F}=(F_{\ff},F_{\rf})$ constant (at least in the critical range) when varying 
in $(\wy, \he)$. Then $T_0$ is called a trace Bessel kernel, $T_0\in \Psi^{-\infty, \calF}_b(F^2\times \R^+)$, 
if it satisfies:
\begin{enumerate}
\item $T_0(\wt,z,\wz;\widetilde{y}, \widehat{\eta})$ is rapidly decreasing as $\wt\to \infty$, locally uniformly in $(z,\wz)$;
\item $T_0(\wt,z,\wz;\widetilde{y}, \widehat{\eta})$ admits a polyhomogeneous 
expansion as $\wt\to 0$, uniformly in the other variables.
\end{enumerate}
\end{defn}

The class of representable trace operators consists of those operators $T\in \Psi^{-\infty, k, F_{\rf}}_e(T_e)$,
whose normal operator $N(T)$, defined as the restriction of $\rho_{\ff}^{1+b-k}\kappa_T$ to $\ff$, is given by
\begin{align}\label{N-trace}
N_k(T_0) := \int_{\R^b} e^{iY\eta} T_0(|\eta|, z,\wz; \wy, \widehat{\eta}) |\eta|^{-k+1} \, \dbar \eta,
\end{align}
for some $T_0\in \Psi^{-\infty, \calF}_b(F^2\times \R^+)$. An example is a trace operator 
$\textup{Op}(T_0)\in \Psi^{-\infty, k, F_{\rf}}_e(T_e)$, defined on test functions $u$ supported near $\del M$ by 
\begin{align*}
\left[\textup{Op}(T_0) u \right] (y,z):=\int e^{i(y-\widetilde{y})\eta} 
T_0(\wx,z,\wz; y, \widehat{\eta}) \circ \kappa_{|\eta|}^{-1} u(\wx, \wy,\wz) 
|\eta|^{-k} \dbar \eta \, d\wx \, d\wy \, d\wz.
\end{align*}
and extended trivially away from the singular neighborhood. The corresponding operator kernel is given in local coordinates by
\begin{equation}
\begin{split}
K_{\textup{Op}(T_0)}(y,z, \wx, \wy, \wz) &= 
\int_{\R^b} e^{i(y-\widetilde{y})\eta} T_0(\wx|\eta|,z,\wz; y, \widehat{\eta}) |\eta|^{-k+1} \dbar \eta \\
&= \wx^{-1-b+k}  \int_{\R^b} e^{iY\eta} T_0(|\eta|,z, \wz; \wy+\wx Y, \widehat{\eta}) |\eta|^{-k+1} \, \dbar \eta \\
&=  \wx^{-1-b+k} N_k(T_0) + \mathcal{O}(\wx^{-b+k}).
\end{split}
\label{kernel-trace}
\end{equation}

%%%%%%%%%%%%%%%%%%%
\subsection{Representable Poisson operators}
%%%%%%%%%%%%%%%%%%%%
Finally, we introduce the Bessel Poisson kernels.
\begin{defn}\label{potential-Bessel}
Let $P_0(t,z,\wz;\widetilde{y}, \widehat{\eta})$ be polyhomogeneous on $\R^+\times F^2$ with index set 
$\mathcal{J}=(J_{\lf}, J_{\ff})$, parametrized and varying smoothly in $(\widetilde{y}, \widehat{\eta})\in S^*B$. 
Then $P_0$ is called a Bessel Poisson operator, $P_0\in \Psi^{-\infty, \mathcal{J}}_b(\R^+\times F^2)$, if: 
\begin{enumerate}
\item $P_0(t,z,\wz;\widetilde{y}, \widehat{\eta})$ is rapidly decreasing as $t\to \infty$, 
locally uniformly in $(z,\wz)$;
\item $P_0(t,z,\wz;\widetilde{y}, \widehat{\eta})$ admits a polyhomogeneous expansion as $t\to 0$, 
uniformly in the other variables.
\end{enumerate}
\end{defn}

The representable Poisson operators are operators $P\in \Psi^{-\infty, k, J_{\lf}}_e(P_e)$ with leading coefficient 
at the front face, the normal operator $N(P)$, given by 
\begin{align}\label{N-potential}
N_k(P_0) := \int_{\R^b} e^{iY\eta} P_0(|\eta|, z,\wz; \wy, \widehat{\eta}) |\eta|^{-k+1} \, \dbar \eta
\end{align}
for some $P_0\in \Psi^{-\infty, J_{\lf}, J'_{\ff}}_b(\R^+ \times F)$. If $\textup{Op}(P_0)\in \Psi^{-\infty, k, J_{\lf}}_e(P_e)$
is defined near $\partial M$ by 
\begin{align*}
\left[\textup{Op}(P_0) u \right] (x,y,z):=\int e^{i(y-\widetilde{y})\eta} \kappa_{|\eta|} 
\circ P_0(x,z,\wz; y, \widehat{\eta}) u(\wy,\wz) |\eta|^{-k+1} \, \dbar \eta \, d\wy \, d\wz, 
\end{align*}
then the Schwartz kernel is given locally by 
\begin{equation}
\begin{split}
K_{\textup{Op}(P_0)}(x, y, z, \wy, \wz) &= 
\int_{\R^b} e^{i(y-\widetilde{y})\eta} P_0(x|\eta|,z,\wz; y, \widehat{\eta}) |\eta|^{-k+1} \dbar \eta \\
&= x^{-1-b+k}  \int_{\R^b} e^{iY\eta} P_0(|\eta|, z, \wz; \wy+xY, \widehat{\eta}) |\eta|^{-k+1} \dbar \eta, \\
&= x^{-1-b+k} N_k(P_0) + \mathcal{O}(x^{-b+k}).
\end{split}
\label{kernel-potential}
\end{equation}

%%%%%%%%%%%%%%%%%%%%%%%%%%%
\subsection{Composition of representable operators}
%%%%%%%%%%%%%%%%%%%%%%%%%%%
We conclude this section by proving that the property of being representable is closed under composition.

\smallskip

\noindent{\bf Residual \ $\circ$\ Poisson: } 
Let $\textup{Op}(G_0)$ and $\textup{Op}(P_0)$ be a residual edge and an edge Poisson operator associated to the 
Bessel operator $G_0$ Bessel Poisson operator $P_0$, respectively.  Using \eqref{kernel-edge} and \eqref{kernel-potential},
the composition $\textup{Op}(G_0) \circ \textup{Op}(P_0)$ is given by 
\begin{equation}
\begin{split}
K_{\textup{Op}(G_0) \circ \textup{Op}(P_0)}&(x,y,z, \wy, \wz)  \\
=\int\int &e^{i(y-y')\eta} G_0(x|\eta|,z,\wx |\eta|, z'; y, \widehat{\eta}) |\eta|^{-g+1} \, \dbar \eta \\
&e^{i(y'-\wy)\eta'} P_0(\wx|\eta'|,z',\wz; y', \he') |\eta'|^{-p+1} \, \dbar \eta' \, d\wx \, dy' \, dz' \\
=\int\int &e^{i(Y-Y')\eta} G_0(|\eta|,z, t |\eta|, z'; \wy+xY, \widehat{\eta}) x^{-1-b+g}|\eta|^{-g+1} \, \dbar \eta \\
&e^{iY'\eta'} P_0(t|\eta'|,z',\wz; \wy+xY', \he') x^{p} |\eta'|^{-p+1} \, \dbar \eta' \, dt \, dY' \, dz',
\end{split}
\end{equation}
where we have substituted
\begin{align}
Y=\frac{y-\wy}{x}, \ Y'=\frac{y'-\wy}{x}, \ t=\frac{\wx}{x}.
\end{align}
Replacing $Y'$ by $(-Y')$ we find for the leading $x^{-1-b+(p+g)}$ coefficient
\begin{equation}
\label{edge-potential}
\begin{split}
N(\textup{Op}& (G_0) \circ \textup{Op}(P_0)) = \int e^{iY\eta}G_0(|\eta|,z, t |\eta|, z'; \wy, \widehat{\eta}) |\eta|^{-g+1}   \\
& \times \int e^{i(\eta-\eta')Y'} P_0(t|\eta'|,z',\wz; \wy, \he') |\eta'|^{-p+1} \, \dbar \eta'\, dY' \, \dbar \eta \, dt\, dz'  \\ 
&=\int e^{iY\eta} (G_0\circ P_0)(|\eta|, z, \wz; \wy, \widehat{\eta})|\eta|^{-p-g+1} \, \dbar \eta = N_g(G_0)\circ N_p(P_0).
\end{split}
\end{equation}
This proves that the normal operator of this composition is representable and has the form \eqref{N-potential}. 

\smallskip
\noindent{\bf Poisson\ $\circ$\ trace:} 
Now consider a Poisson operator $\textup{Op}(P_0)$ associated to the Bessel Poisson kernel $P_0$ and a trace operator 
$\textup{Op}(T_0)$ associated to the Bessel trace kernel $T_0$. The composition in \eqref{kernel-potential} and 
\eqref{kernel-trace} takes the form 
\begin{equation}
\begin{split}
K_{\textup{Op}(P_0) \circ \textup{Op}(T_0)}&(x,y,z, \wx, \wy, \wz)  \\
=\int\int &e^{i(y-y')\eta} P_0(x|\eta|,z, z'; y, \widehat{\eta}) |\eta|^{-p+1} \, \dbar \eta \\
&e^{i(y'-\wy)\eta'} T_0(\wx|\eta'|,z',\wz; y', \he') |\eta'|^{-\tau+1} \, \dbar \eta' \, dy' \, dz' \\
=\int\int &e^{i(Y-Y')\eta} P_0(s|\eta|,z, z'; \wy+\wx Y, \widehat{\eta}) x^{-1-b+p}|\eta|^{-p+1} \, \dbar \eta \\
&e^{iY'\eta'} T_0(|\eta'|,z',\wz; \wy+\wx Y', \he') x^{\tau} |\eta'|^{-\tau+1} \, \dbar \eta' \, dY' \, dz',
\end{split}
\end{equation}
where
\begin{align}
Y=\frac{y-\wy}{\wx}, \ Y'=\frac{y'-\wy}{\wx}, \ s=\frac{x}{\wx}.
\end{align}

As before, substituting $Y'$ by $(-Y')$, we obtain 
\begin{equation}
\label{potential-trace}
\begin{split}
N(\textup{Op}& (P_0) \circ \textup{Op}(T_0)) =\int e^{iY\eta}P_0(s|\eta|,z, z'; \wy, \widehat{\eta}) |\eta|^{-p+1} \\
&\times \int e^{i(\eta-\eta')Y'} T_0(|\eta'|,z',\wz; \wy, \he') |\eta'|^{-\tau+1} \, \dbar \eta'\, dY' \, \dbar \eta \, dz' \\ 
&=\int e^{iY\eta} (P_0\circ T_0)(s|\eta|, z, |\eta|, \wz; \wy, \widehat{\eta})|\eta|^{-p-\tau+1} \, \dbar \eta = N_p(P_0)\circ N_{\tau}(T_0),
\end{split}
\end{equation}
so this composition is again representable. 

\smallskip

\noindent{\bf Trace\ $\circ$\ Poisson:}
Finally, if $\textup{Op}(T_0)$ is a trace operator associated to the Bessel trace kernel $T_0$ and $\textup{Op}(P_0)$ 
is a Poisson operator associated to the Bessel Poisson kernel $P_0$, then \eqref{kernel-trace} and \eqref{kernel-potential} 
becomes 
\begin{equation}
\begin{split}
K_{\textup{Op}(T_0) \circ \textup{Op}(P_0)}&(y,z, \wy, \wz)  \\
=\int\int &e^{i(y-y')\eta} T_0(x|\eta|,z, z'; y, \widehat{\eta}) |\eta|^{-\tau+1} \, \dbar \eta \\
&e^{i(y'-\wy)\eta'} P_0(x|\eta'|,z',\wz; y', \he') |\eta'|^{-p+1} \, \dbar \eta' \, dx\, dy' \, dz' \\
=\int\int &e^{i(Y-Y')\eta} T_0(t|\eta|,z, z'; \wy+r Y, \widehat{\eta}) r^{-1-b+\tau}|\eta|^{-\tau+1} \, \dbar \eta \\
&e^{iY'\eta'} P_0(t|\eta'|,z',\wz; \wy+r Y', \he') r^{p} |\eta'|^{-p+1} \, \dbar \eta' \, dt\, dY' \, dz',
\end{split}
\end{equation}
where 
\begin{align}
Y=\frac{y-\wy}{r}, \ Y'=\frac{y'-\wy}{r}, \ t=\frac{x}{r}.
\end{align}
Substituting $Y'$ by $(-Y')$, and taking the limit $r\to 0$, we obtain the principal symbol of a 
pseudodifferential operator on the closed manifold $B$ acting on sections of the trace bundle: 
\begin{equation}
\label{trace-potential}
\begin{split}
N(\textup{Op}& (T_0) \circ \textup{Op}(P_0)) =\int e^{iY\eta}T_0(t|\eta|,z, z'; \wy, \widehat{\eta}) |\eta|^{-\tau+1} \\ 
&\times \int e^{i(\eta-\eta')Y'} P_0(t|\eta'|,z',\wz; \wy, \he') |\eta'|^{-p+1} \, \dbar \eta'\, dY' \, \dbar \eta \, dx\, dz' \\ 
&=\int e^{iY\eta} (T_0\circ P_0)(z, \wz; \wy, \widehat{\eta})|\eta|^{-p-\tau+1} \, \dbar \eta = N_{\tau}(T_0)\circ N_p(P_0).
\end{split}
\end{equation}

%%%%%%%%%%%%%%%%%%%%
\section{Trace and Poisson operators of an elliptic edge operator}\label{trace-pot}
%%%%%%%%%%%%%%%%%%%%
Let $L\in \textup{Diff}^m_e(M)$ be an elliptic differential edge operator. We use all the same notation as above, and
assume, in particular, that $B(L)$ is injective on $t^{\ud}L^2$ and surjective on $t^{\ld}L^2$.  

Define
\begin{equation}
\calH_{\ld,\ud}(L) = \{ u \in x^{\ld}L^2: Lu \in x^{\ud} L^2\}.
\label{calH}
\end{equation}
We often refer to this as $\calH_{\ld,\ud}$, or even just $\calH$. This is a Hilbert space with respect the graph norm
\[
||u||_{\calH} = ||u||_{ x^{\ld} L^2} + ||Lu||_{ x^{\ud} L^2}.
\]
In this section we define and study the trace map, which assigns to any $u \in \calH_{\ld,\ud}$ the set of
leading coefficients in its expansion with exponents between $\ld$ and $\ud$. We also construct the 
Poisson operator for $L$, which assigns to an appropriate set of leading coefficients an element 
of $\ker L \cap \calH_{\ld,\ud}$. 

A subtlety in these definitions is that leading coefficients are sections of the trace bundle 
\[
\calE(L) = \bigoplus_{j=0}^N  \calE(L; \zeta_j) 
\]
introduced in \S 2. A standing assumption in this paper is that the $\zeta_j$ are independent of $y \in B$, and because of
this, the different subbundles $\calE(L; \zeta_j)$ do not interact with one another.  Thus, to simplify the notation
in this section, we suppose that there is only a single indicial root $\zeta_0 \in \mathfrak{S}(L)$, and we 
$\calE(L) = \calE(L; \zeta_0)$. 

%%%%%%%%%%%%%%%%%%%%
\subsection{The trace map for the model Bessel operator}
%%%%%%%%%%%%%%%%%%%%
The model Bessel operator corresponding to $L$ is 
\[
B_{\wy, \he}(L)=\sum\limits_{j+|\A|+|\beta|\leq m}
a_{j,\A,\beta}(0,\wy,z)(t\partial_t)^j(it\he)^{\A}\partial_z^{\beta}, 
\]
which acts (as an unbounded operator) on $t^{\ld}H^m_{\mathrm{loc}}(\RR^+ \times F; dt\, dz)$.  Just as for $L$, however, 
we are primarily interested in its restriction to 
\[
\calH_{\ld, \ud}^B = \{\w\in t^{\ld}L^2(\R^+\times F; dt\, dz):  \  B(L)\w\in t^{\ud}L^2\}.
\]

If $\w \in \calH^B_{\ld,\ud}$, then we can follow the same strategy as in the proof of Proposition~\ref{explsoln}
to obtain the (strong) expansion
\[
\w \sim \sum_{\ell \geq 0} \sum_{p=0}^{p_0} t^{-i\zeta_0+\ell}(\log t)^p \, \w_{\ell,p}(\wy, z) + \tilde{\w}, 
\quad \tilde{\w} \in \bigcap_{\e > 0} t^{\ud - \e}L^2. 
\]
Indeed, writing $B(L) = I(L) + E$, where $E$ contains all terms with `extra' powers of $t$, then $B(L) \w = f$
becomes $I(L) \w = f - E \w$. The term $E \w$ creates new `higher order' terms $t^{-i\zeta_0 + \ell}$ with 
$\ell > 0$, but discarding these  we obtain
\begin{equation}
I_{\wy}(L) \left( \sum_{p=0}^{p_0} t^{-i\zeta_0}( \log t)^p \, \w_{0,p}(\wy, z) \right) = 0.
\label{may}
\end{equation}
By definition of the fibres of the trace bundle, this expression in parentheses lies in $\calE_{\wy}(L; \zeta_0)$ for each $\wy$. 

Now consider how this expansion varies as a function of $\wy$. Even if $f$ depends smoothly on $\wy$, 
the individual coefficients $\w_{0,p}$ may fail to be smooth (or even continuous) in $\wy$ because the order 
$p_0$ of the indicial root may vary.  This is where the properties of the trace bundle from \cite{KM},
discussed above in \S 2, become crucial. As explained there, on any neighbourhood 
$\calU \subset B$ over which $\calE$ is trivialized, there exist smooth functions $\phi_{\wy,k}(t,z)$, 
$k \leq  m_0 = m(\zeta_0)$, such that 
\[
\sum_{p=0}^{p_0} t^{-i\zeta_0}(\log t)^p \, \w_{0,p}(\wy, z) = \sum_{k=1}^{m_0} f_k(\wy) \phi_{\wy,k}(t,z),
\]
where, somewhat remarkably, $f_k \in \calC^\infty(\calU)$ even though the number of terms in the sum on the left 
may be discontinuous. 

Using all of this, we can now state the 
\begin{defn}
The Bessel trace map $\mathrm{Tr}_{B(L)}$ is the operator which assigns to each $\w \in \calC^\infty(S^*B; \calH_{\ld,\ud}^B)$ 
a section of $\calE(L; \zeta_0)$ which is represented in a neighbourhood $\calU \subset B$ 
in which $\calE(L; \zeta_0)$ is trivialized by the smooth basis of sections $\phi_{\wy,k}$ by the 
$m_0$-tuple $\{f_1, \ldots, f_{m_0}\}$. 
\end{defn}
Note that if $\w(\wy,\he) \in t^{\ud}L^2$ for each $(\wy, \he)$, then $\Tr_{B(L)} \w =0$. 

\begin{prop}
The operator $\mathrm{Tr}_{B(L)}$ is a representable Bessel trace kernel in the sense of Definition~\ref{trace-Bessel}. 
\end{prop}
\begin{proof}
Recall the definition of the trace bundle in Proposition \ref{trace-bundle}. 
Then for a solution $\w$, the singular part of its Mellin transform is a section of $\calE(L)$.
Consider, following \cite{KM}, the Hilbert space adjoint $I_{\wy}(L)(\zeta)^*$ of the indicial operator pencil and set 
$I_{\wy}(L)^*(\zeta):=I_{\wy}(L)(\bar{\zeta})^*$. This depends smoothly on $y$ and is a holomorphic family of
Fredholm operators in $\zeta\in \C$. Its indicial roots are the complex conjugates of elements of 
$\textup{Spec}_b(L)$. We denote its trace bundle by $\calE^*(L)$.  This suggestive notation is vindicated
by a central result in \cite[Theorem 5.3]{KM}, which asserts the nondegeneracy of the pairing 
\begin{equation}
\begin{split}
&\calE_{\wy}(L)(\zeta_0) \times \calE^*_{\wy}(L)(\bar{\zeta_0}) \rightarrow \C, \\
&[\phi, \psi]:= \frac{1}{2\pi} \oint_{B_\epsilon(\zeta_0)} \phi(\zeta) I_{\wy}(L)^*(\bar{\zeta}) \psi (\bar{\zeta}) \, d\zeta.
\end{split}
\end{equation}
for any sufficiently small $\epsilon >0$. Identifying $\calE_{\wy}(L, \zeta_0)$ with the kernel of $I_{\wy}(L)$ on the space of 
finite combinations $\sum a_q(z) t^{-i\zeta_0} \log^q (t), a_j \in C^\infty(F)$, we may assign to each basis element 
$\phi_{\wy,j}$ its dual, $\phi^*_{\wy,j}$, with respect to this pairing. If $\chi\in \calC^\infty_0(\RR)$ is a cutoff function 
which equals one near $0$, then the integral kernel of the Bessel trace map is
\begin{equation}
\begin{split}
\mathrm{Tr}_{B(L)}(t,\wz; \wy) &= \frac{1}{2\pi} \bigoplus_{j=1}^{m_0} 
 \oint_{B_\epsilon(\zeta_0)} t^{i\zeta -1} \chi(t)
I_{\wy}(L)^*(\bar{\zeta}) \phi^*_{\wy,j} (\bar{\zeta}, \wz) \, d\zeta \\
&:= \bigoplus_{j=1}^{m_0} \Phi^*_j(t,\wy,\wz).
\end{split}
\end{equation}
This satisfies the conditions of Definition \ref{trace-Bessel}, and hence $\mathrm{Tr}_{B(L)}$ is a representable Bessel 
trace kernel. 

The absence of the variable $z$ in this formula is a result of the identification of the asymptotic coefficients 
of $\w$ with local sections of the trace bundle, since this bundle is trivialized by the smooth basis $\{\phi_{\wy,j}\}$,
which has coefficients $\{f_1, ..., f_{m_0}\}$ depending only on $\wy$. 
\end{proof}

%%%%%%%%%%%%%%%%%%%%
\subsection{Trace of solutions to the normal operator}
%%%%%%%%%%%%%%%%%%%%
The next step is to carry out a similar analysis of the trace operator for the normal operator $N(L)$. 
Recall that $N(L)$ is identified with the restriction of the lift $\beta^*L$ to the front face in $M^2_e$ 
with respect to the blowdown map $\beta:M^2_e \to M^2$, and in the projective coordinates $(s,Y,z)$ 
from \eqref{proj-coord} this takes the form \eqref{normalop} (with $Y$ replacing the variable $u$ there). 
The normal operator is equivalent to the Bessel operator \eqref{Bessel} through Fourier transform (in $Y$) 
and rescaling (setting $s=t/|\eta|$): 
\[
\mathscr{F}\circ N_{\wy}(L) \circ \mathscr{F}^{-1} \mid_{s=t/|\eta|} = 
B_{\wy,\he}(L).
\]
Thus if $\w \in s^{\ld}L^2(ds\, dY\, dz)$ is such that $N(L)\w\in s^{\ud}L^2$, then its Fourier transform $\widehat{\w}$
evaluated at $s = t/|\eta|$ is an element of $\calH^B_{\ld,\ud}$. As such, it can be written locally as
\[
\widehat{\w} (s, \eta, z) = \sum_{k=1}^{m_0} a_k(\wy, \eta) \phi_{\wy,k}(s|\eta|, z).
\]
We define the trace map for $N_{\wy}(L)$ as 
\begin{align*}
\textup{Tr}_{N(L)} \w &:= \bigoplus_{j=1}^{m_0} \int e^{i(Y-\widetilde{Y})\eta} \Phi^*_j(s|\eta|,\wy, \wz) 
\, \w (s, \widetilde{Y}, \wz) |\eta|^{-i\zeta_0+1}\, ds\, \dbar\eta\, d\widetilde{Y}\, d\wz\\
&= \bigoplus_{j=1}^{m_0} \int_{\R^b} e^{iY\eta} a_j(\wy, \eta) |\eta|^{-i\zeta_0} \,\dbar \eta\in H^{-(\Im\zeta_0 -\ld +1/2)}(\RR^b, dY) \otimes \calE_{\wy}(L;\zeta_0).
\end{align*}
where we used the regularity result \cite[Thm. 7.3]{M}. 
From \eqref{N-trace} and since $\textup{Tr}_{B(L)}$ 
is a Bessel trace kernel,  we infer that 
\begin{align}\label{normal-trace}
\textup{Tr}_{N(L)}(Y, \wy, \wz) = \int_{\R^b} e^{iY\eta} \textup{Tr}_{B(L)} (|\eta|, \wy, \wz) |\eta|^{-i\zeta_0+1}\, \dbar \eta,
\end{align}
is smooth on the front face of $T_e$ and polyhomogeneous at the boundaries of this face.

%%%%%%%%%%%%%%%%%%%%
\subsection{The trace map of $L$}\label{trace-bounded}
%%%%%%%%%%%%%%%%%%%%
The construction above determines a Schwartz kernel representation for a trace map of the operator $L$ itself. 
Indeed, following \eqref{normal-trace}, define in local coordinates of the corner neighborhood in $M^2$
\[
\textup{Tr}_L(\wx, y, \wy, \wz) := \int_{\R^b} e^{i\eta (y-\wy)} \textup{Tr}_{B(L)} (x|\eta|, \wy, \wz) |\eta|^{-i\zeta_0+1}\, \dbar \eta, 
\]
and extend smoothly to the interior. From the work above,
\[
\textup{Tr}_L : \calH_{\ld, \ud} \to H^{-(\Im\zeta_0 -\ld +1/2)}(B,\calE(L;\zeta_0))
\]
is a bounded mapping, a representable trace operator. Note that this operator $\mathrm{Tr}_L$ is by no means unique. 

%%%%%%%%%%%%%%%%%%%%
\subsection{The edge Poisson operator}
%%%%%%%%%%%%%%%%%%%%
We define the Bessel Poisson operator
\begin{align*}
P_0:\calE(L,\zeta_0) \to \calH^B_{\ld,\ud}, \quad (f_1, ...,f_{m_0}) \mapsto \sum_{j=1}^{m_0} f_j \phi_{\wy, j},
\end{align*}
with the integral kernel (as before $\wz$ is absent)
\[
P_0(t,z;\wy) = \bigoplus\limits_{j=1}^{m_0} \phi_{\wy, j}(t,z).
\]
In particular, $P_0$ is representable in the sense of Definition \ref{potential-Bessel}, with the index set 
$J_{\lf}$ determined by the asymptotic expansion of each $\phi_{\wy, j}$. The associated normal operator
is given by
\[
N_{-i\zeta_0+1}(P_0) = \int_{\R^b} e^{iY\eta} P_0(|\eta|, z; \wy) |\eta|^{i\zeta_0}\, \dbar \eta,
\]
which we extend off the front face to define an edge Poisson operator 
\begin{align*}
\textup{Op}(P_0): H^{-(\Im(\zeta_0) -\ld +1/2)}(B, \calE(\zeta_0)) 
\longrightarrow x^{\ld}H^{\infty}_e(M,g).
\end{align*}
Consider the orthogonal projector (cf. \cite{M})
\begin{align}
P_1: x^{\ld}L^2(M,g)\to \ker L\cap x^{\ld}L^2(M,g),
\end{align} 
which is a residual edge operator discussed in \S \ref{edge-pseudos}, 
with the corresponding edge Bessel kernel $P_{01}$, which is the 
Schwartz kernel of the orthogonal projection of $t^{\ld}L^2(dt\, dz)$ 
onto  $\ker B(L)(\wy, \widehat{\eta})\cap t^{\ld}L^2(dt\, dz)$. We define
\begin{align}
P_L=P_1\circ \textup{Op}(P_0).
\end{align}
By the composition rule \eqref{edge-potential} we find
\begin{align*}
N(P_L)=N_0(P_1) \circ N_1(P_0)= \int e^{iY\eta} (P_{01}\circ P_0)(|\eta|, z) \, \dbar \eta.
\end{align*}
The restriction of Bessel trace map $\textup{Tr}_{B(L)}$ to $\ker B(L) \cap t^{\ld}L^2$ is injective, 
since $B(L)$ is injective on $t^{\ud}L^2$. Hence $\textup{Tr}_{B(L)}$ admits a left-inverse $\textup{Tr}_{B(L)}^{-1}$,
mapping $\calE_{\wy}(L,\zeta_0)$ to $\ker B(L) \cap t^{\ld}L^2$, which is a true inverse when restricted to 
$\textup{im}\, \textup{Tr}_{B(L)}(\ker B(L) \cap t^{\ld}L^2)$.

\begin{lemma}\label{bessel-P-T}
$\textup{Tr}_{B(L)}^{-1} = P_{01} \circ P_0 \restriction \textup{im}\, \textup{Tr}_{B(L)}(\ker B(L) \cap t^{\ld}L^2).$
\end{lemma}
\begin{proof}
Note that by \cite[(5.8)]{M} there exists a generalized inverse $G_0$ such that 
$G_0B(L)=I-P_{01}$. Consequently, for any $\w \in \calH^B_{\ld,\ud}$ we find 
$\w- P_{01}\w =G_0B(L)\w \in t^{\ud}L^2$. Hence, $\textup{Tr}_{B(L)}\w = \textup{Tr}_{B(L)}P_{01} \w$. Thus
\[
\textup{Tr}_{B(L)} P_{01} \circ P_0 (\textup{Tr}_{B(L)}\w) = \textup{Tr}_{B(L)} P_0(\textup{Tr}_{B(L)}\w) = \textup{Tr}_{B(L)} \w.
\]
If $\w\in \ker B(L) \cap t^{\ld}L^2$, then $\w = P_{01}\circ P_0 (\textup{Tr}_{B(L)} \w)$ since $B(L)$ is
injective on $t^{\ud}L^2$.
\end{proof}

\begin{prop} $(N(P_L)\circ \textup{Tr}_{N(L)}) \w=\w$, for $\w\in \ker N(L)\cap s^{\ld}L^2$.
\end{prop}

\begin{proof}
We compute according to \eqref{potential-trace}
\begin{align*}
N(P_L)\circ \textup{Tr}_{N(L)}= \int e^{iY\eta}(P_{01} \circ P_0 \circ \textup{Tr}_{B(L)(\wy,\he)}) 
(s|\eta|, \widetilde{s}|\eta|, z, \wz) \, \dbar \eta.
\end{align*}
which is the normal operator of a residual edge operator. Consider $\w(s,Y,z)\in \ker N(L)\cap s^{\ld}L^2$. 
As before, $\widehat{\w}(t/|\eta|,\eta,z) \in \ker B(L)(\wy, \widehat{\eta})\cap t^{\ld}L^2$. 
Thus we compute by Lemma \ref{bessel-P-T}
\begin{align*}
&N(P_L)\circ \textup{Tr}_{N(L)}\, \w \\ &=\int e^{i(Y-\widetilde{Y})\eta}
(P_{01} \circ P_0 \circ \textup{Tr}_{B(L)(\wy,\he)}) 
(s|\eta|, \widetilde{s}|\eta|, z, \wz) \w(\widetilde{s}, \widetilde{Y}, \wz)\, \dbar \eta\, d\widetilde{s}\, d\widetilde{Y}\, d\wz\\
&=\int e^{iY\eta} (P_{01} \circ P_0 \circ \textup{Tr}_{B(L)(\wy,\he)}) 
(s|\eta|, \widetilde{s}|\eta|, z, \wz)\, \widehat{\w}
(\widetilde{s},\eta,\wz) \, \dbar \eta\, d\widetilde{s}\, d\wz \\ &= \int e^{iY\eta}\,  \widehat{\w}(s,\eta, z) \, 
\dbar \eta= \w(s,Y,z).
\end{align*}
\end{proof}

%%%%%%%%%%%%%%%%%%%%
\section{Fredholm theory of elliptic edge boundary value problems}\label{fredholm}
%%%%%%%%%%%%%%%%%%%%
We return now to the general situation, where $\mathfrak{S}(L) = \{\zeta_0,..,\zeta_N\}$. Fix a collection $E_1,..,E_M$ of finite 
rank vector bundles over $B$ and set $E = \oplus_{k=1}^M E_k$. Now consider the collection of classical pseudodifferential operators 
\begin{align*}
&Q_{kj} \in \Psi^{d_k - \Im (\zeta_j)}(B; \calE(L; \zeta_j), E_k), \quad j=1,..,N, \ k=1,..,M, \\
&Q_{kj}: H^s(B, \calE(L, \zeta_j)) \to H^{s-d_k+\Im (\zeta_j)}(B,E_k), \ s\in \R.
\end{align*}
Define the homogeneity rescalings
\begin{align*}
&\eta(L): \calE \to \calE, \ 
(u_1,..,u_N) \mapsto (|\eta|^{\Im\zeta_1}u_1,..,|\eta|^{\Im\zeta_N}u_N), \\
&\eta(Q): E\to E, \ (e_1,..,e_M)\mapsto (|\eta|^{d_1}e_1,..,|\eta|^{d_M}e_M).
\end{align*}
The matrix $(Q_{kj})$ defines the pseudodifferential system $Q$ where
\[
\sigma_0(Q)(\wy,\eta)=\eta(Q)\circ \sigma_0(Q)(\wy,\widehat{\eta})\circ \eta(L)^{-1}. 
\]
(Note that $\eta$ appears on the left and $\he = \eta/|\eta|$ on the right.)

We now recall the form of the general edge boundary value problem: 
\begin{defn}
Let $L\in \textup{Diff}^m_e(M)$ be edge elliptic, and suppose that $Q = (Q_{kj})$ is
as above. Then the edge boundary value problem $(L,Q)$ is the set of equations
\begin{align*}
& Lu=f\in x^{\ud}L^2(M), \\
& Q(\textup{Tr}_L\, u)=\phi \in \bigoplus\limits_{k=1}^M H^{\ld -d_k-1/2}(B,E_k). 
\end{align*}
for $u\in x^{\ld}H^m_e(M)$.
\end{defn}
We have already stated, in Definition~\ref{typesofbcs}, the definitions of right-, left- and full ellipticity
of the boundary problem $(L,Q)$. 

Clearly
\begin{equation}
\begin{split}
(L,Q): & \calH \to x^{\ud}L^2(M,g)\oplus \left(\bigoplus\limits_{k=1}^M H^{\ld-d_k-1/2}(B,E_k)\right), \\
& u\mapsto (Lu, Q(\textup{Tr}_L\, u)).
\end{split}
\label{mainmap}
\end{equation}
is continuous.  Our goal is to show that it is semi-Fredholm if $(L,Q)$ satisfies conditions i) or ii) of Definition~\ref{typesofbcs}, 
and Fredholm if $(L,Q)$ satisfies condition iii).  This is proved by a parametrix construction. 

%%%%%%%%%%%%%%%%%%%%%%%%%
\subsection{The right-elliptic case} 
%%%%%%%%%%%%%%%%%%%%%%%%%
Consider a right-elliptic system $(L,Q)$. Since 
\begin{align*}
\sigma(Q)\restriction \textup{im}\, 
\textup{Tr}_{B(L)}:\textup{im}\, \textup{Tr}_{B(L)} \to E
\end{align*}
is surjective, there exists a right parametrix 
\begin{align*}
K:\bigoplus\limits_{k=1}^MH^{\ld-d_k-1/2}(B,E_k)\to \bigoplus\limits_{j=1}^N H^{\ld -\Im (\zeta_j)-1/2}(B,\calE(L; \zeta_j))
\end{align*}
for $Q$; this has principal symbol 
\begin{equation}\label{K}
\begin{split}
&\sigma(K)(\wy, \eta)=\eta(L)\circ \sigma(Q)^{-1}(\wy, \widehat{\eta})\circ \eta(Q)^{-1}, \\
&\sigma(Q)^{-1}(\wy, \widehat{\eta}): E_{\wy} \to \textup{im}\, \textup{Tr}_{B(L)(\wy,\he)},
\end{split}
\end{equation}
where $\sigma(Q)^{-1}(\wy, \widehat{\eta})$ is some choice of right-inverse for $\sigma(Q)(\wy,\widehat{\eta})
\restriction \textup{im}\, \textup{Tr}_{B(L)}$ which varies smoothly in $(\wy, \widehat{\eta})$. 
\begin{thm}\label{right}
If $(L,Q)$ is right-elliptic, then \eqref{mainmap} is semi-Fredholm, with closed range of finite codimension. 
A right parametrix for it is given by 
\begin{align*}
\mathcal{G}(f,\phi)= Gf + P_L[K(\phi - Q(\textup{Tr}_{L}Gf))],
\end{align*}
where $G$ is the generalized inverse for $L$ on $x^{\ld}H^m_e(M)$. 
\end{thm}
\begin{proof}
By definition, $LG = \mbox{Id} - P_2$, where $P_2$ is the orthogonal projection onto the finite-dimensional space 
$\textup{coker}L\cap x^{\ld}L^2$. Thus if $f\in x^{\ud}L^2$, then 
\begin{align*}
\|Gf\|_\calH&=\|Gf\|_{x^{\ld}H^m_e}+ \|LGf\|_{x^{\ud}L^2}\\
&\leq \|Gf\|_{x^{\ld}H^m_e}+ \|f\|_{x^{\ud}L^2} + \|P_2f\|_{x^{\ud}L^2}.  
%\\&\leq \|Gf\|_{x^{\ld}H^m_e}+ \|f\|_{x^{\ud}L^2} + c\cdot \|P_2f\|_{x^{\ld}L^2} \leq 
\end{align*}
Since $G$ is bounded on $x^{\ld}L^2$ and $||f||_{x^{\ld}L^2} \leq ||f||_{x^{\ud}L^2}$, we have 
$||Gf||_{\calH} \leq C (||f||_{x^{\ud}L^2} + ||P_2 f||_{x^{\ld}L^2})$.  Hence $G: \ker P_2 \cap x^{\ud}L^2 \rightarrow \calH$ 
is bounded. For simplicity below, we assume that $P_2 \equiv 0$; if this projector is nontrivial, it only changes
things by a finite dimensional amount, which does not affect any of the Fredholmness statements below. 

Next, both 
\begin{align}\label{bded2}
P_L:  \bigoplus_{j=0}^N H^{-(\Im (\zeta_j) -\ld +1/2)}(B, \calE(L, \zeta_j)) \to \ker L \cap x^{\ld}H^{\infty}_e(M) \subset \calH
\end{align}
and 
\begin{equation}\label{bded3}
\textup{Tr}_{L}: \calH \to  \bigoplus_{j=0}^N H^{-(\Im (\zeta_j) -\ld +1/2)}(B, \calE(L,\zeta_j)),
\end{equation}
are continuous, the latter by the discussion in \S \ref{trace-bounded}. All of this, together with continuity of the
pseudodifferential operators $Q$ and $K$ between appropriate Sobolev spaces over $B$, shows that the
parametrix $\calG$ is a bounded mapping.

We now compute the error term $((L,Q)\mathcal{G} - \mbox{Id})(f,\phi)$. 
Since $L P_L = 0$, and we are assuming that the cokernel of $L$ is trivial, we have $L \calG (f,\phi) = LGf = f$. 
Next, 
\begin{align*}
Q\, \textup{Tr}_{L}\, \mathcal{G}(f,\phi)&= Q\left[\textup{Tr}_{L}\, Gf+ 
\textup{Tr}_{L}\, P_L(K(\phi-Q(\textup{Tr}_{L}\, Gf)))\right]\\
&= Q\textup{Tr}_{L}\, Gf+ (Q\circ \textup{Tr}_{L}\circ P_L\circ K)(\phi-Q(\textup{Tr}_{L}\, Gf))\\
&=\phi +  (Q\circ \textup{Tr}_{L}\circ P_L\circ K-I)(\phi-Q(\textup{Tr}_{L}\, Gf)).
\end{align*}
It thus remains to prove that 
\[
(Q\circ \textup{Tr}_{L}\circ P_L\circ K-I): \bigoplus\limits_{k=1}^M 
H^{\ld-d_k-1/2}(B,E_k) \to \bigoplus\limits_{k=1}^MH^{\ld-d_k-1/2}(B,E_k)
\]
is compact. This is however simply a pseudo-differential operator over the closed manifold $B$, so it suffices to check that
its principal symbol vanishes. We compute, using \eqref{trace-potential}, that
\begin{align*}
&\sigma_0(Q\circ \textup{Tr}_{L}\circ P_L\circ K-I)(\wy,\eta)\\
&= \sigma(Q)(\wy, \eta)\circ (\textup{Tr}_{B(L)}\circ P_{01} \circ P_0)\circ \sigma(K)(\wy, \eta)-I.
\end{align*}
By definition, $\sigma(K)$ maps into $\calC_{B(L)}$, so all terms cancel and this principal symbol vanishes. 
This completes the proof. 
\end{proof}

%%%%%%%%%%%%%%%%%%%%%%%
\subsection{Left-elliptic edge boundary value problem} 
%%%%%%%%%%%%%%%%%%%%%%%
Now consider a set of boundary operators $Q$ which satisfy the left-elliptic conditions. 
Since $\sigma_Q(\wy, \he) \restriction \calC_{B(L)}$ is injective, there exists a matrix of pseudodifferential operators
\[
K:\bigoplus\limits_{k=1}^M H^{\ld-d_k-1/2}(B,E_k)\to 
\bigoplus\limits_{j=1}^N H^{\ld -\Im (\zeta_j)-1/2}(B, \calE(L; \zeta_j)),
\]
with principal symbol 
\[
\sigma(K)(\wy, \eta)=\eta(L)\circ \sigma(Q)^{-1}(\wy, \widehat{\eta})\circ \eta(Q)^{-1}
\]
where 
\[
\sigma(Q)^{-1}(\wy, \widehat{\eta}): E_{\wy} \to \calC_{B(L)}, 
\]
is a left-inverse to $\sigma(Q)(\wy,\he)\restriction \calC_{B(L)}$. Note that $K$ is 
not necessarily a left-parametrix for $Q$, since $\sigma(Q)^{-1}(\wy, \he)$ does not invert the full symbol
$\sigma(Q)(\wy,\he)$, but this is not required for our argument.
\begin{thm}\label{left}
If $(L,Q)$ is left-elliptic, then 
\[
(L,Q): \calH \to x^{\ud}L^2(M)\oplus \left(\bigoplus\limits_{k=1}^MH^{\ld-d_k-1/2}(B,E_k)\right),
\]
is semi-Fredholm with left parametrix
\begin{align*}
\calG(f,\phi)= Gf + P_L[K(\phi - Q(\textup{Tr}_{L}\, Gf))].
\end{align*}
\end{thm}
\begin{proof}
As before, $\calG$ is a bounded operator and we compute for any $u\in \calH$, 
\begin{align*}
\mathcal{G}(L,Q)u&= GLu + P_L[K(Q\, \textup{Tr}_{L}\, u-Q\textup{Tr}_{L}\, GLu)]\\
&=GLu + (P_L\circ K \circ Q \circ \textup{Tr}_{L})(u-GLu)\\
&=u + (P_L\circ K \circ Q \circ \textup{Tr}_{L}-I) P_1 u,
\end{align*}
where $P_1$ is the orthogonal projection onto the nullspace of $L$ in $x^{\ld}L^2$. Hence we must show that
$(P_L\circ K \circ Q \circ \textup{Tr}_{L}-I)\circ P_1$ is compact on $\calH$.   

By the form of $||\cdot ||_{\calH}$ and since $LP_L=0$ and $LP_1=0$, we need only check compactness of 
\[
(P_L\circ K \circ Q \circ \textup{Tr}_{L}-I)\circ P_1: \calH \longrightarrow \ker L \cap \calH. 
\]
By the composition results in \S \ref{triple}, $(P_L\circ K \circ Q \circ \textup{Tr}_{L})$ is an edge
operator of order $-\infty$, i.e.\ has no diagonal singularity, and has normal operator
\begin{multline*}
N(P_L\circ K \circ Q \circ \textup{Tr}_{L}) = \\ \int e^{iY\eta} 
(P_{01} \circ P_0 \circ \sigma(K) \circ \sigma(Q) \circ \textup{Tr}_{B(L)})
(s|\eta|, |\eta|, z, \wz; \wy, \he) \, \dbar \eta,
\end{multline*}
whence 
\begin{multline*}
N((P_L\circ K \circ Q \circ \textup{Tr}_{L}-I)\circ P_1) =  \\
\int e^{iY\eta} (P_{01} \circ P_0 \circ \sigma(K) \circ \sigma(Q) \circ \textup{Tr}_{B(L)} \circ P_{01} 
-P_{01})(s|\eta|, |\eta|, z, \wz; \wy, \widehat{\eta})\, |\eta| \, \dbar \eta.
\end{multline*}
In this combination, $\sigma(Q)(\wy, \he)$ acts on $\calC_{B(L)}$, so that 
$\sigma(Q)(\wy, \he)$ and $\sigma(Q)^{-1}(\wy, \he)$ cancel. After further obvious cancellations, 
this normal operator reduces to
\[
\int e^{iY\eta} (P_{01}-P_{01})(s|\eta|, |\eta|, z, \wz; \wy, \widehat{\eta})
\, |\eta| \, \dbar \eta=0.
\]
Finally, using the boundedness properties of $P_1$, $P_L$ and $\Tr_L$, we see that
\[
R:=(P_L\circ K \circ Q \circ \textup{Tr}_{L}-I)\circ P_1: 
\calH\longrightarrow \ker L \cap \, x^{\ld} H^\infty_e(M) \hookrightarrow x^{\ld}L^2
\]
is bounded as well. From the composition results in \S \ref{triple}, 
$R\in \Psi^{-\infty, 0, E_{\lf}, E_{\rf}}(M^2_e)$ and $N(R)=0$, so in fact $R\in \Psi^{-\infty, 1, E_{\lf}, E_{\rf}}(M^2_e)$,
with index sets 
\begin{equation}
\begin{split}
E_{\lf} &=\{(\zeta,p)\in \textup{Spec}_b(L) \mid \Im \zeta) >\ld -1/2\}, \\
E_{\rf} &= \{(\zeta,p) \in \C\times \N_0 \mid (\zeta +2\ld, p)\in E_{\lf}\},
\end{split}
\end{equation} 
see \eqref{P-index}. Its compactness is now a consequence of \cite[Prop. 3.29]{M}.
\end{proof}

From Theorems \ref{right} and \ref{left} we now conclude the 
\begin{cor}
Let $(L,Q)$ be elliptic. Then 
\[
(L,Q): \calH_{\ld,\ud} \to x^{\ud}L^2(M)\oplus \left(\bigoplus\limits_{k=1}^MH^{\ld-d_k-1/2}(B,E_k)\right),
\]
is Fredholm, with parametrix
\begin{align*}
\mathcal{G}(f,\phi)= Gf + P_L[K(\phi - Q(\textup{Tr}_{L}Gf))].
\end{align*}
\end{cor}

We conclude by presenting one simple application of this machinery. 
\begin{prop}
Let $u\in x^{\ld} L^2(M)$ and suppose that $Lu=0$ and $\Tr_L u = 0$. Then $u \in x^{\ud}H^\infty_e(M)$.
\end{prop}
\begin{proof}
Choose any left elliptic boundary value problem $(L,Q)$, and let $\calG$ be its left parametrix, as constructed above, so
that $\calG \circ (L,Q) = \mbox{Id} - \calR$. Then $\Tr_L u = 0$, so $u = \calR u$. Since $N(\calR)=0$, \cite[Thm. 3.25]{M} gives 
that 
\[
\calR: x^{\delta}H^s_e(M)\to x^{\delta+\epsilon}H^{\infty}_e(M), \ s\geq 0
\]
is bounded for some $\epsilon >0$ which depends only on $\calR$. Iterating this statement gives the result. 
\end{proof}

%----------------------------------------------------------------
\bibliographystyle{amsplain}

\end{document}